\documentclass[11pt]{amsart}
\usepackage[latin1]{inputenc}
\usepackage{amsmath}
\usepackage{amsfonts}
\usepackage{amssymb,amsbsy,amsthm}
\usepackage{hyperref}
\usepackage{geometry}
\usepackage{color}
\usepackage[T1]{fontenc}
\usepackage[enableskew,vcentermath]{youngtab}
\usepackage{ytableau}
\usepackage{tikz}

\newcommand{\Q}{\mathcal Q}
\newcommand{\f}{\mathcal F}
\newcommand{\TTT}{\mathcal T}
\newcommand{\symm}{S}
\newcommand{\C}{\mathcal C}
\newcommand{\A}{\mathcal A}
\newcommand{\bgamma}{\bar\gamma}

\def\x#1{(\textnormal{mod}\ x^{#1})}

\DeclareMathOperator\Des{Des}

\DeclareMathOperator\cDes{cDes}
\DeclareMathOperator\CDes{CDes}
\DeclareMathOperator\maj{maj}
\DeclareMathOperator\cyc{cyc}
\DeclareMathOperator\sgn{sgn}
\def\CC{{\mathbb C}}

\def\ZZ{{\mathbb Z}}

\newcommand\comp[2]{\alpha{(#2,#1)}}
\newcommand\ccomp[2]{\alpha^{\cyc{(#2,#1)}}}

\newcommand{\cs}{\tilde s}

\newcommand{\ch}{{\operatorname{ch}}}

\def\SYT{{\rm SYT}}

\theoremstyle{plain}
\newtheorem{theorem}{Theorem}[section]
\newtheorem{proposition}[theorem]{Proposition}
\newtheorem{lemma}[theorem]{Lemma}
\newtheorem{corollary}[theorem]{Corollary}
\newtheorem{conjecture}[theorem]{Conjecture}
\newtheorem{problem}[theorem]{Problem}
\theoremstyle{definition}
\newtheorem{definition}[theorem]{Definition}
\newtheorem{defn}[theorem]{Definition}
\newtheorem{example}[theorem]{Example}

\newtheorem{remark}[theorem]{Remark}
\newtheorem{observation}[theorem]{Observation}

\newcommand{\card}[1]{\ensuremath{\left|#1\right|}}

\newcommand{\todo}[1]{\vspace{5 mm}\par \noindent
	\marginpar{\textsc{ToDo}} \framebox{\begin{minipage}[c]{0.9
				\textwidth}
			#1 \end{minipage}}\vspace{5 mm}\par}

\begin{document}
	
\title[Higher Lie Characters and Cyclic Descents on Conjugacy Classes]
{Higher Lie characters and\\ cyclic descent extension on conjugacy classes}


\author{Ron M.\ Adin}
\address{Department of Mathematics, Bar-Ilan University, 
	Ramat-Gan 52900, Israel}
\email{radin@math.biu.ac.il}

	\author{P\'al Heged\"us}
	\address{Department of Algebra, Institute of Mathematics, Budapest University of Technology and Economics, M\H uegyetem rkp. 3., H-1111 Budapest, Hungary}
	\email{hegpal@math.bme.hu}


\author{Yuval Roichman}
\address{Department of Mathematics, Bar-Ilan University, 
Ramat-Gan 52900, Israel}
\email{yuvalr@math.biu.ac.il}

\date{February 9, 2023}
	
\thanks{PH was partially supported by Hungarian National Research, Development and Innovation Office (NKFIH) Grant No.~K115799 and by Bar-Ilan University visiting grant. The project leading to this application has received funding from the European Research Council (ERC) under the European Union's Horizon 2020 research and innovation programme, Grant agreement No.\ 741420. 
RMA and YR were partially supported by the Israel Science Foundation, grant no.\ 1970/18 and 
by an MIT-Israel MISTI grant. 
}
	
\maketitle	
	
\begin{abstract}
A now-classical cyclic extension of the descent set of a permutation has been introduced by Klyachko and Cellini.
Following a recent axiomatic approach to this notion, it is natural to ask which sets of permutations admit such a (not necessarily classical) extension.

The main result of this paper is
a complete answer in the case of conjucay classes of permutations.
It is shown that the conjugacy class of cycle type $\lambda$ 
has such an extension if and only if 
$\lambda$ is not of the form $(r^s)$ for some square-free $r$.
The proof involves a detailed study of hook constituents in
higher Lie characters.
\end{abstract}

	

\tableofcontents

\section{Introduction}

\subsection{Background and main result}

The study of descent sets for permutations may be traced back to Euler. 
A cyclic extension of this classical concept was introduced in the study of
Lie algebras~\cite{Klyachko} and descent algebras~\cite{Cellini}. Surprising connections of the cyclic descent
notion to a variety of mathematical areas were found later.
	
The {\em descent set} of a permutation $\pi = [\pi_1, \ldots, \pi_n]$ in the symmetric group $\symm_n$ on $n$ letters is
\[
\Des(\pi) := \{1 \le i \le n-1 \,:\, \pi_i > \pi_{i+1} \}
\quad \subseteq [n-1],
\]
where $[m]:=\{1,2,\ldots,m\}$.  
Cellini~\cite{Cellini} introduced a natural notion of {\em cyclic descent set}: 
\[
\CDes(\pi) := \{1 \leq i \leq n \,:\, \pi_i > \pi_{i+1} \}
\quad \subseteq [n],
\]
with the convention $\pi_{n+1}:=\pi_1$.  
The more restricted notion of cyclic descent number had been used previously by Klyachko~\cite{Klyachko}.
This cyclic descent set was further studied by Dilks, Petersen and Stembridge~\cite{DPS} and others.

\smallskip
	
There exists a well-established notion of descent set for standard Young tableaux (SYT), but it has no obvious cyclic analogue. 
In a breakthrough work, Rhoades~\cite{Rhoades} defined a notion of cyclic descent set for standard Young tableaux of rectangular shape.
The properties common to Cellini's definition (for permutations) and Rhoades' construction (for SYT) appeared in other combinatorial settings as well~\cite{PPR, Pechenik, ER, AER}.
This led to an abstract definition~\cite{ARR}, as follows.
%
%
%
	
\begin{definition}\label{def:cDes}\cite{ARR}
Let $\TTT$ be a finite set, equipped with a set valued map (called {\em descent map}) 
$\Des: \TTT \longrightarrow 2^{[n-1]}$. 
Let ${\rm shift}: 2^{[n]} \longrightarrow 2^{[n]}$ be the mapping on subsets of $[n]$ induced by the cyclic shift $i \mapsto i + 1 \pmod n$ of elements $i \in [n]$.
A {\em cyclic extension} of $\Des$ is
a pair $(\cDes,p)$, where 
$\cDes: \TTT \longrightarrow 2^{[n]}$ is a map 
and $p: \TTT \longrightarrow \TTT$ is a bijection,
satisfying the following axioms:  for all $T$ in  $\TTT$,
\[
\begin{array}{rl}
\text{(extension)}   & \cDes(T) \cap [n-1] = \Des(T),\\
\text{(equivariance)}& \cDes(p(T))  = {\rm shift}(\cDes(T)),\\
\text{(non-Escher)}  & \varnothing \subsetneq \cDes(T) \subsetneq [n].\\
\end{array}
\]
\end{definition}
	

The term ``non-Escher'' refers to M.\ C.\ Escher's drawing ``Ascending and Descending", which illustrates the impossibility
of the cases $\cDes(\pi) = \varnothing$ and $\cDes(\pi) = [n]$ for permutations $\pi \in S_n$.

		For connections of cyclic descents to 
		Kazhdan-Lusztig theory see~\cite{Rhoades}; for topological aspects and connections to the  
		Steinberg torus see~\cite{DPS}; for twisted Sch\"utzenberger promotion see~\cite{Rhoades, Huang}; for cyclic quasisymmetric functions and Schur-positivity see~\cite{AGRR};	for Postnikov's toric Schur functions see~\cite{ARR}.
		The goal of this paper is to determine which conjugacy classes of the symmetric group carry a cyclic descent extension.

\bigskip

Cellini's cyclic descent set, denoted $\CDes$, is a special case of a cyclic descent extension, denoted in general $\cDes$, as attested by the following observation. 
		
\begin{observation}
Let $\Des$ and $\CDes$ denote the classical descent set and Cellini's cyclic descent set on permutations, respectively.
Let $p: S_n \rightarrow S_n$ be the rotation 
\[
\begin{array}{rcl}
[\pi_1,\pi_2,\ldots,\pi_{n-1},\pi_n]&\overset{p}{\longmapsto}&
[\pi_n,\pi_1,\pi_2,\ldots,\pi_{n-1}]. 
\end{array}
\]
Then the pair $(\CDes,p)$ is a cyclic descent extension  of $\Des$  
on $S_n$ in the sense of Definition~\ref{def:cDes}.
\end{observation}
	

Unlike the full symmetric group, for many conjugacy classes,   
Cellini's definition does not provide a cyclic extension, see Section~\ref{sec:Cellini} below.

\begin{example}
Consider the conjugacy class of 4-cycles in $S_4$,
\[
\C_{(4)}=\{2341, 2413, 3142, 3421, 4123, 4312\}.
\]
Cellini's cyclic descent sets are
\[
\{3\}, \{2,4\}, \{1,3\}, \{2,3\}, \{1\}, \{1,2\},
\]
respectively; this family is not closed under cyclic rotation.
On the other hand, redefining the cyclic descent sets to be
\[
\cDes(2341)=\{3,4\},\ \cDes(2413)=\{2,4\},
\  \cDes(3142)=\{1,3\}, 
\]
\[
\cDes(3421)=\{2,3\},\ \cDes(4123)=\{1,4\},\ \cDes(4312)=\{1,2\}
\]
and defining the map $p$ by
\[
2341 \rightarrow 4123 \rightarrow 4312 \rightarrow 3421 \rightarrow 2341
\]
and
\[
3142 \rightarrow 2413 \rightarrow 3142,
\]
the pair $(\cDes,p)$ does determine a cyclic extension of $\Des$ for this conjugacy class.
\end{example}
	
The goal of this paper is to show that most conjugacy classes in $S_n$ carry a cyclic descent extension.
In fact, we obtain a full characterization.

\smallskip

Recall that an integer is \emph{square-free} if no prime square divides it; in particular, $1$ is square-free. Our main result is 
	
\begin{theorem}\label{thm:main}
Let $\lambda$ be a partition of $n$, and let
$\C_\lambda \subseteq S_n$ be the corresponding conjugacy class. 
The descent map 
$\Des$ 
on $\C_\lambda$  has a cyclic extension $(\cDes,p)$ if and only if 
$\lambda$ is not of the form $(r^s)$ for some square-free $r$.
\end{theorem}
	
	

\subsection{Proof method} 
	
The proof of Theorem~\ref{thm:main} is non-constructive and   
involves a detailed study of the hook constituents in higher Lie characters.

\subsubsection{Higher Lie characters}

\begin{definition}\label{def:higher}
For a partition $\lambda$ of $n$, let $\C_\lambda$ be the conjugacy class consisting of all the permutations in $S_n$ of cycle type $\lambda$,
and let $\chi^\lambda$ denote the irreducible $S_n$-character corresponding to $\lambda$.
Let $Z_\lambda$ be the centralizer of a permutation in $\C_\lambda$ (defined up to conjugacy). 
If $k_i$ denotes the number of parts of $\lambda$ equal to $i$, then $Z_\lambda$ is isomorphic to the direct	product $\times_{i=1}^n \ZZ_i\wr S_{k_i}$.
Here and in the rest of the paper $\ZZ_i$ denotes the cyclic group of order $i$.
	
For each $i$, let $\omega_i$ be the linear character on $\ZZ_i\wr S_{k_i}$ indexed by the
$i$-tuple of partitions $(\emptyset, (k_i), \emptyset,
\dots,\emptyset)$. 
In other words, let $\zeta_i$ be a primitive irreducible character on the cyclic group $\ZZ_i$, and extend it to the wreath product $\ZZ_i\wr S_{k_i}$ so that it is homogeneous on the base subgroup $\ZZ_i^{k_i}$ and trivial on the wreathing subgroup $S_{k_i}$. Denote this extension by $\omega_i$.
Now let
\[
\omega^\lambda := \bigotimes_{i=1}^n \omega_i,
\]
a linear character on $Z_\lambda$. 
Define the corresponding {\em higher Lie character} to be the induced character
\[
\psi^\lambda := \omega^\lambda\uparrow_{Z_\lambda}^{S_n}.
\]
\end{definition}

The study of higher Lie 
characters can be traced back 
to Schur~\cite{schur}. 
An old problem of Thrall~\cite{thrall} 
is to provide an explicit combinatorial interpretation of the 
multiplicities of the irreducible characters in the higher Lie character,
see also~\cite[Exercise 7.89(i)]{EC2}.
Only partial results are known: the case $\lambda = (n)$ was solved by Kraskiewicz and Weyman~\cite{KraskiewiczWeyman};
D\'esarm\'enien and Wachs~\cite{DesarmenienWachs} resolved a coarser version of Thrall's problem 
for the sum of higher Lie characters over all  
derangements, see also~\cite{ReinerWebb}. 
The best result so far is
Schocker's expansion~\cite[Theorem~3.1]{schocker}, which however involves signs and rational coefficients. 
For recent discussions see, e.g., \cite{Reiner, AS, Sundaram}.

A remarkable theorem of  Gessel and Reutenauer~\cite[Theorem 2.1]{GR} 
applies higher Lie characters to describe the fiber sizes of the descent set map on conjugacy classes.  
Their proof applies an interpretation of higher Lie character $\psi^\lambda$ in terms of quasisymmetric functions 
(Theorem~\ref{thm:GR22} below). 
It follows that 
higher Lie characters can be used to prove the existence of cyclic descent extensions 
as explained below.

\subsubsection{Hook multiplicities and cyclic descent extensions}\ \\

Recall the standard notation $s_\lambda$ for the Schur function indexed by a partition $\lambda$, 
as well as $\f_{n,D}$ for the fundamental quasisymmetric function indexed by a subset $D\subseteq [n-1]$;
see Definition~\ref{def:fundamental_QSF}.
A subset $\A\subseteq S_n$ is {\em Schur-positive} 
if the
associated quasisymmetric function
\[
\Q(\A):=\sum\limits_{a\in \A}\f_{n,\Des(a)},
\]
is symmetric and Schur-positive.

For an integer $0 \le k<n$ and a Schur-positive subset $\A \subseteq S_n$ denote
\[
m_{k,\A} := \langle \Q(\A), s_{(n-k,1^k)} \rangle,
\]
where $s_{(n-k,1^k)}$ is the Schur function indexed by the hook partition $(n-k,1^k)$.

First we prove the following key lemma, which provides an algebraic criterion for the existence of a cyclic descent extension. 

\begin{lemma}\label{lem:21}
A Schur-positive set $\A\subseteq S_n$ has a cyclic descent extension if and only if 
the following two conditions hold:
\[
\begin{array}{rl}
	\text{\rm (divisibility)}   & 
	\text{the polynomial }
	\sum_{k=0}^{n-1} m_{k,\A} x^k
	\text{ is divisible by } 1+x; \\
	\text{\rm (non-negativity)} & 
	\text{the quotient has nonnegative coefficients.}
\end{array}
\]
\end{lemma}		

See Lemma~\ref{lem:2}  below.

\subsubsection{Divisibility}\ \\

By the 
Gessel-Reutenauer theorem, 
for every 
conjugacy class $\C_\lambda$ 
the quasisymmetric function 
$\Q(\C_\lambda)$ is the Frobenius image of the higher Lie character $\psi^\lambda$, thus  $\C_\lambda$ is  Schur-positive; 
see Theorem~\ref{thm:GR22} below.

\medskip	

For a partition $\lambda \vdash n$ denote
\begin{equation}\label{notation:m_k_hLc}
m_{k,\lambda} := m_{k,\C_\lambda}
= \langle \Q(\C_\lambda), s_{(n-k,1^k)} \rangle
= \langle \psi^\lambda, \chi^{(n-k,1^k)} \rangle.
\end{equation}
Next we prove 
    
\begin{proposition}\label{prop:divisibility} 
The hook-multiplicity generating function of the higher Lie character $\psi^\lambda$
\[
M_\lambda(x):=\sum\limits_{k=0}^{n-1}m_{k,\lambda}x^k
\]
is divisible by $1+x$ 
if and only if $\lambda$ is not of the form $(r^s)$ for a square-free integer $r$.
\end{proposition}

This divisibility condition is proved using an explicit evaluation of the higher Lie character on $n$-cycles.; see Section~\ref{sec:divisibility} below.
    
\subsubsection{Non-negativity}\ \\

In order to prove Theorem~\ref{thm:main}, it remains to show that the coefficients of the quotient  $M_\lambda(x)/(1+x)$ are nonnegative, whenever $\lambda$ is not of the form $(r^s)$ for a square-free $r$. 
It turns out that partitions $\lambda$ with more than one cycle length are the easiest to handle. 
In that case, a factorization of the associated higher Lie character $\psi^\lambda$ is applied to prove

\begin{lemma}\label{lem:distinct}
Let $\lambda = \mu \sqcup \nu$ be a disjoint union of nonempty partitions with no common part. 
Then
\begin{equation}
\frac{M_\lambda(x)}{1+x} = M_\mu(x) M_\nu(x),     
\end{equation}
and its coefficients are thus non-negative.
\end{lemma}	


%
%
The core of the proof of Theorem~\ref{thm:main} is 
the case of $\lambda = (r^s)$.
For a fixed positive integer $r$, consider the formal power series
\[
M_r(x,y) := \sum_{\substack{i \ge 0 \\ s \ge 1}} m_{i,(r^s)} x^i y^s,
\]
where $m_{i,(r^s)}$ is the multiplicity of the hook character $\chi^{(rs-i,1^i)}$ in the higher Lie character $\psi^{(r^s)}$. 
The following theorem 
completes the proof of Theorem~\ref{thm:main}.






\begin{theorem}\label{t:main_hook-alternating2}
    If $r$ is not square-free then the formal power series 
    \[
    \frac{M_r(x,y)}{1+x}
    \] 
	has non-negative integer coefficients.
\end{theorem}

\begin{proof}[Proof of Theorem~\ref{thm:main}.]
    Combine Lemma~\ref{lem:21}, Proposition~\ref{prop:divisibility}, Lemma~\ref{lem:distinct} and Theorem~\ref{t:main_hook-alternating2}. 
\end{proof}

\subsection{Outline}

The rest of the paper is organized as follows. 

Necessary background, including a cyclic analogue of the Gessel-Reutenauer Theorem, is given in Section~\ref{sec:background}.

The role of hooks in the study of cyclic descent extensions is explained in Section~\ref{sec:hooks}.
In particular, a necessary and sufficient criterion
for Schur-positive sets to carry a cyclic descent extension (Lemma~\ref{lem:21}) is proved in Subsection~\ref{sec:hook_criterion}.
Using this criterion, the proof of Theorem~\ref{thm:main} is reduced
to Proposition~\ref{prop:divisibility} (divisibility) and  Theorem~\ref{t:main_hook-alternating2} (non-negativity).
Proposition~\ref{prop:divisibility} is proved in Subsection~\ref{sec:divisibility}.  

The proof of Theorem~\ref{t:main_hook-alternating2}, stating the non-negativity of the coefficients of $M_\lambda(x)/(1+x)$ whenever this quotient is a polynomial, spans 
Sections~\ref{sec:distinct}, \ref{sec:cycles} and~\ref{sec:Witt}:
the case of more than one cycle length is considered in Section~\ref{sec:distinct};
the case of $n$-cycles 
is considered in Section~\ref{sec:cycles};
and the case of cycle type $(r^s)$ with $s>1$ is considered in Section~\ref{sec:Witt}.
In the case of more than one cycle length, 
non-negativity is proved using a factorization of the associated higher Lie character (Lemma~\ref{lem:distinct}).  
In the case of $n$-cycles,  
combining a combinatorial formula for inner products with a variant of the Witt transform 
proves unimodality of the sequence of hook-multiplicities (Proposition~\ref{prop:s=1_unimodal}). This, in turn, implies the desired non-negativity of the quotient.  


In Section~\ref{sec:Witt} we lift the $n$-cycle result to the case of cycle type $(r^s)$ with $s>1$. 
In Subsection~\ref{sec:e} we provide an explicit expression for the coefficients of $(1+x)M_r(x,y)$, see Theorem~\ref{prop:e_i_formuli}.
This expression is used to obtain a product formula 
for the bivariate polynomial $1+(1+x)M_r(x,y)$
in Subsection~\ref{sec:prod}, 
and to deduce Theorem~\ref{t:main_hook-alternating2} in Subsection~\ref{sec:nn}. 

Additional results are presented in Section~\ref{sec:final}. 
In Subsection~\ref{remarks_on_s=1} it is shown that 
Lemma~\ref{lem:char_eval_cycles} and Theorem~\ref{prop:e_i_formuli} 
imply well-known combinatorial identities. 
In Subsection~\ref{sec:Cellini} it is shown that the natural approach does not provide a cyclic descent extension for conjugacy classes in $S_n$.
Palindromicity of the hook-multiplicity generating function $M_{(r^s)}(x)$ is studied in Subsection~\ref{sec:palindrom}. 

Section~\ref{sec:open} concludes the paper with final remarks and open problems.



\bigskip

\noindent
{\bf Acknowledgements.} The second author is grateful for the hospitality of Bar-Ilan University during his Sabbatical leave.
Special thanks are due to Jan Saxl, whose personality influenced the development of this paper in many ways; in particular, the paper is inspired by the  
ingenious use of a higher Lie character in~\cite{Saxletal}.
Thanks also 
to Eli Bagno, Jonathan Bloom, Sergi Elizalde, Darij Grinberg, 
Vic Reiner, Richard Stanley, Sheila Sundaram and Josh Swanson for their comments.

\section{Preliminaries}\label{sec:background}

	The role of quasisymmetric 
	in the study of 
	the distribution of descent 
	sets is discussed in Subsection~\ref{sec:prel1}. 
	The results presented  here are 
	used in Section~\ref{sec:hooks} 
	to establish 
	the reduction of existence of cyclic descent extension on conjugacy classes to the study of hook-multiplicities in higher Lie characters.
	In Subsection~\ref{sec:prel2} we  
present cyclic analogues of a few classical results. These analogues are used for  
certain enumerative applications; the reader may skip this subsection.

	
	\subsection{Quasisymmetric functions and descents}\label{sec:prel1} 
	A symmetric function is called {\em Schur-positive} if all the
	coefficients in its expansion in the basis of Schur functions are non-negative. Recall the notation 
		 $s_{\lambda/\mu}$ for the Schur function indexed by a 
		 skew shape $\lambda/\mu$. 
	
	\begin{defn}\label{def:fundamental_QSF}
	For each subset $D \subseteq [n-1]$ define the {\em
		fundamental quasisymmetric function}
	\[
	\f_{n,D}({\bf x}) := \sum_{\substack{i_1\le i_2 \le \ldots \le i_n \\ {i_j < i_{j+1} \text{ if } j \in D}}} 
	x_{i_1} x_{i_2} \cdots x_{i_n}.
	\]
	\end{defn}
		The following theorem is due to Gessel.
		
		Denote the set of standard Young tableau of skew shape $\lambda/\mu$
by $\SYT(\lambda/\mu)$.
There is an established notion of descent set for $\SYT(\lambda/\mu)$
\begin{equation}\label{def:Des_SYT}
\Des(T) := \{1 \le i \le n-1 \,:\,
            i+1 \text{ appears in a lower row of }T\text{ than }i\}.
\end{equation}

\smallskip
		
\begin{theorem}{\rm \cite[Theorem 7.19.7]{EC2}}\label{G1} 
For every skew shape $\lambda/\mu$,
\[
\sum\limits_{T\in \SYT(\lambda/\mu)}\f_{n,\Des(T)}=s_{\lambda/\mu}.
\]
\end{theorem}

		Given any subset $\A\subseteq S_n$, define the
		quasisymmetric function
		\[
		\Q(\A) := \sum\limits_{\pi\in \A} \f_{n,\Des(\pi)}.
		\]

		Finding subsets of permutations $\A\subseteq S_n$, for which $\Q(\A)$ is symmetric (Schur-positive), is a  long-standing problem, see~\cite{GR}. 

We write $\lambda\vdash n$ to denote that $\lambda$ is a partition of a positive integer $n$.  
For $D\subseteq [n-1]$ let ${\bf x}^D:=\prod\limits_{i\in D}x_i$.


\begin{lemma}\label{lem:Schur-gf}
For every subset $\A\subseteq S_n$ and a family $\{c_\lambda\}_{\lambda \vdash n}$
of coefficients, the equality
\begin{equation}\label{eq:l1}
\Q(\A) =\sum_{\lambda\vdash n} c_{\lambda} s_{\lambda}
\end{equation}
is equivalent to the equality
\begin{equation}\label{eq:l2}
\sum\limits_{a\in \A} {\bf x}^{\Des(a)} = \sum_{\lambda\vdash n} c_{\lambda} 
\sum_{T\in \SYT(\lambda)}{\bf x}^{\Des(T)}.
\end{equation}
\end{lemma}

\begin{proof}
By Theorem~\ref{G1}, 
Equation~\eqref{eq:l1} is equivalent to
\[
\sum\limits_{a\in \A} \f_{n,\Des(a)}= \sum\limits_{\lambda\vdash n} c_{\lambda} \sum\limits_{T\in \SYT(\lambda)}\f_{n,\Des(T)}.
\]
Next recall from~\cite[Ch. 7]{EC2} that the fundamental quasisymmetric functions in $x_1,\dots,x_n$
form a basis of the vector space ${\rm{QSym}}_n$ of quasisymmetric functions in $n$ variables.
Finally, apply the vector space isomorphism from ${\rm{QSym}}_n$ 
to the space of 
square-free polynomials in $x_1,\ldots,x_n$, which maps 
$\f_{n,D}$ to ${\bf x}^D$. 
\end{proof}

\begin{corollary}\label{cor:Schur-gf}
For every finite family $S$ of skew shapes of size $n$  and every subset $\A\subseteq S_n$ 
$\Q(\A)= 
\sum\limits_{\lambda/\mu\in S} c_{\lambda/\mu} s_{\lambda/\mu}$ 
if and only if 
\[
\sum\limits_{a\in \A} {\bf x}^{\Des(a)}=\sum\limits_{\lambda/\mu\in S} c_{\lambda/\mu} 
\sum\limits_{T\in \SYT(\lambda/\mu)}{\bf x}^{\Des(T)}. 
\]
\end{corollary}

\medskip


	Corollary~\ref{cor:Schur-gf} 
	will be combined with the following Theorem~\ref{conj1} to provide criteria for the existence of cyclic descent extensions for Schur-positive sets, see proof of Lemma~\ref{lem:2} and Remark~\ref{rem:distinct} below.  

A {\em ribbon} is a skew shape ``of width $1$'', namely: containing no $2 \times 2$ rectangle.
The following theorem was proved in~\cite{ARR} using Postnikov's toric symmetric functions.
	
\begin{theorem}[{\cite[Theorem 1.1]{ARR}}]\label{conj1}
Let $\lambda/\mu$ be a skew shape with $n$ cells.
The descent map $\Des$ on $\SYT(\lambda/\mu)$ has a cyclic extension $(\cDes,p)$ if and only if $\lambda/\mu$ is not a connected ribbon.
\end{theorem}
	
A constructive proof was recently given by 
Huang~\cite{Huang}.

\medskip

Let $\lambda \vdash n$ be a partition of $n$ and let $\psi^\lambda$ be {\em the higher Lie character} indexed by $\lambda$ (see Definition~\ref{def:higher}).


The following result is proved in 
~\cite{GR}.

	\begin{theorem}\label{thm:GR22}\cite[Proof of Theorem 2.1]{GR}
		For every partition $\lambda$ of $n\ge 1$, 
		\[
		\Q(\C_\lambda) = 
		\ch(\psi^\lambda),
		\]
		where $\ch$ is the Frobenius characteristic map.
Equivalently,		
		\[
		\Q(\C_\lambda)=
		\sum\limits_{\mu \vdash n} 
		\langle \psi^\lambda, \chi^\mu  \rangle s_\mu.
		\]
		In particular, $\Q(\C_\lambda)$ is Schur-positive.
	\end{theorem}


\subsection{Cyclic analogues}
\label{sec:prel2}
		


The following theorem is 
due to Gessel.

For a subset $J=\{j_1  < \ldots < j_t\} \subseteq [n-1]$, the composition (of $n$)
\[
\comp{n}{J}:= (j_1,j_2-j_1,j_3-j_2,\ldots,j_t-j_{t-1},n-j_t)
\]
defines a {\em connected ribbon} having the entries of $\alpha(J,n)$ as row lengths, from bottom to top. 
Let $s_{\alpha(J,n)}$ be the associated {\em (skew) ribbon Schur function}. 


\begin{theorem}\label{thm:fiber-Des}
\cite[an immediate consequence of Theorem 3]{Gessel}
Let  $\A$ be a finite set, equipped with a descent 
map
$\Des: \A \longrightarrow 2^{[n-1]}$. 
If 
\[
\Q(\A):=\sum\limits_{a\in \A} \f_{n,\Des(a)}
\]
is symmetric then
\[
|\{a \in \A:\ \Des(a)=J\}|=\langle s_{\alpha(J,n)}, \Q(\A) \rangle
\qquad (\forall J\subseteq [n-1]).
\]
\end{theorem}

For a subset $\varnothing \ne J = \{ j_1 < j_2 < \ldots < j_t \}
\subseteq [n]$
define the corresponding {\em cyclic composition} of $n$  as
\[
\ccomp{n}{J} := (j_2-j_1, \ldots, j_t - j_{t-1}, j_1 + n - j_t),
\]
with $\ccomp{n}{J} := (n)$ when $J = \{j_1\}$;
note that $\ccomp{n}{\varnothing}$ is not defined.
The corresponding {\em affine (cyclic) ribbon Schur function} was
defined in~\cite{ARR} as
\[
\cs_{\ccomp{n}{J}} := \sum_{\varnothing \neq I \subseteq J}
(-1)^{\#(J \setminus I)} h_{\ccomp{n}{I}}.
\]

\begin{theorem}\label{thm:fiber-cDes}~\cite[Cor. 4.13]{AGRR}
Let $\A$ be a finite set, equipped with a descent 
map
$\Des: \A \longrightarrow 2^{[n-1]}$ 
which has a cyclic  extension. If
\[
\Q(\A):=\sum\limits_{a\in \A} \f_{n,\Des(T)}
\]
is symmetric then the fiber sizes of (any) cyclic descent map satisfy
\[
|\{a \in \A \,:\, \cDes(a) =J\}|=\langle \Q(\A), \cs_{\ccomp{n}{J}}
\rangle
\qquad \left( \forall \varnothing \subsetneq J \subsetneq [n] \right).
\]
\end{theorem}

\begin{proposition}\cite[Lemma 2.2]{ARR} If $A\subseteq S_n$ carries a cyclic descent extension, then 
the cyclic descent set generating function is uniquely determined.
\end{proposition}

Furthermore, 

\begin{corollary}\label{cor:cDes-dist-Schur}
Let $\A\subseteq S_n$ be a symmetric set which carries a cyclic descent extension 
and $S$ be a finite set of skew shapes of size $n$ which are not connected ribbons. Then 
for every cyclic descent extension 
the following equations are equivalent:
\begin{equation}\label{eq:dist1}
\Q(\A) = \sum\limits_{\lambda/\mu\in S} c_{\lambda/\mu} s_{\lambda/\mu},
\end{equation}
and
\begin{equation}\label{eq:dist2}
\sum\limits_{a\in \A} {\bf x}^{\cDes(a)} = \sum\limits_{\lambda/\mu\in S} c_{\lambda/\mu} 
\sum\limits_{T\in \SYT(\lambda/\mu)} {\bf x}^{\cDes(T)}.
\end{equation}
\end{corollary}

\begin{proof}
By Theorem~\ref{thm:fiber-cDes} and Theorem~\ref{G1},
if Equation~\eqref{eq:dist1} holds then,
for every $\varnothing \subsetneq J \subsetneq [n]$,
\begin{align*}
|\{a \in \A \,:\, \cDes(a) = J\}| 
&= \langle \Q(\A), \cs_{\ccomp{n}{J}} \rangle
= \langle \sum\limits_{\lambda/\mu\vdash n} c_{\lambda/\mu} s_{\lambda/\mu}, \cs_{\ccomp{n}{J}} \rangle \\
&= \sum\limits_{\lambda/\mu\vdash n} c_{\lambda/\mu} \langle \Q(\SYT(\lambda/\mu)), \cs_{\ccomp{n}{J}} \rangle \\
&= \sum\limits_{\lambda/\mu\vdash n} c_{\lambda/\mu} |\{T \in \SYT(\lambda/\mu) \,:\, \cDes(T) =J\}|.
\end{align*}
Thus Equation~\eqref{eq:dist2} holds.

For the opposite direction, let $x_n=1$ in Equation~\eqref{eq:dist2} and apply Corollary~\ref{cor:Schur-gf} to deduce  Equation~\eqref{eq:dist1}.
\end{proof}

Theorem~\ref{thm:GR22} and Theorem~\ref{thm:fiber-Des} imply




\begin{theorem}\label{thm:GR}\cite[Theorem 2.1]{GR}
For every 
conjugacy class $\C_\lambda$ of cycle type $\lambda \vdash n$ the
descent set map $\Des$
has fiber sizes given by
\[
|\{\pi \in \C_\lambda :\ 
\Des(\pi)=J\}| = \langle \Q(\C_\lambda),  s_{\alpha(J,n)}\rangle
\qquad (\forall J\subseteq [n-1]).
\]
\end{theorem}

The following cyclic analogue of Theorem~\ref{thm:GR} results from Theorem~\ref{thm:GR22} and Theorem~\ref{thm:fiber-cDes}.




	
		\begin{theorem}\label{thm:main22}
		For every conjugacy class $\C_\lambda$, which carries a cyclic descent set extension,
			all cyclic extensions of the descent set map $\Des$ 
			have fiber sizes given by
			\[
			|\{\pi\in \C_\lambda :\ \cDes(\pi)=J\}|=\langle \Q(\C_\lambda), \cs_{\ccomp{n}{J}} \rangle \qquad (\forall \varnothing \subsetneq J\subsetneq [n]).
			\]
	\end{theorem}

\section{The role of hooks}\label{sec:hooks}

\subsection{Hooks and near-hooks}\label{sec:near_hooks}

It turns out that hooks and near-hooks play a crucial role in the 
study of cyclic descent extensions.

A {\em hook} is a partition with at most one part larger than $1$. Explicitly, it has the form $(n-k,1^k)$ for some $0 \le k \le n-1$. 
A {\em near-hook} of size $n$ is a hook of size $n+1$ with its (northwestern) corner cell removed; see~\cite{AER} for a somewhat more inclusive definition of this notion. 
Equivalently, recall the {\em direct sum} operation on shapes (partitions), denoted $\lambda \oplus \mu$, 
yielding a skew shape having the diagram of $\lambda$ strictly southwest of the diagram of $\mu$, with no rows or columns in common.  
A near-hook of size $n$ is the direct sum of a one-column partition $(1^k)$ and a one-row partition $(n-k)$, for some $0 \le k \le n$. 
For example, 
\[
(1^2) \oplus (5)
\quad = \quad
\young(\hfil,\hfil)
\,\,\, \oplus \,\,\,
\young(\hfil\hfil\hfil\hfil\hfil)
\quad = \quad
\young(:\hfil\hfil\hfil\hfil\hfil,\hfil,\hfil)
\]

Given an $S_n$-character $\phi$, let  
\begin{equation}\label{e_k_interpretation1}
m_{k,\phi}
:= \langle \phi,\chi^{(n-k,1^k)} \rangle 
\quad (0 \le k \le n-1)
\quad \text{and}\quad	
e_{k,\phi} 
:=\langle \phi,\chi^{(1^k)\oplus(n-k)} \rangle
\quad (0 \le k \le n).
\end{equation}
When $\phi$ is understood from the context, we use the abbreviated notations $m_k:=m_{k,\phi}$ and $e_k:=e_{k,\phi}$.

By Pieri's rule~\cite[Theorem 7.5.17]{EC2} 
combined with the (inverse) Frobenius characteristic map,
\begin{align*}
\chi^{(1^k)\oplus (n-k)} 
= \chi^{(1^k)} \chi^{(n-k)}
&= \ch^{-1}(s_{(1^k)}s_{(n-k)})
= \ch^{-1}(s_{(n-k+1,1^{k-1})} + s_{(n-k,1^k)})\\
&= \chi^{(n-k+1,1^{k-1})} + \chi^{(n-k,1^k)} 
\qquad (0<k<n).
\end{align*}
Equivalently,
\[
\chi^{(n-k,1^k)}=
\sum_{i=0}^{k} (-1)^{k-i}\chi^{(1^i)\oplus (n-i)} \qquad (0\le k \le n-1).
\]
Thus the sequences $\{m_k\}_{k=0}^{n-1}$ and $\{e_k\}_{k=0}^{n}$ 
determine each other via the relations
\begin{equation}\label{eq:alternating}
e_k = m_k + m_{k-1}
\quad \text{and} \quad
m_k = \sum_{i=0}^k (-1)^{k-i} e_i
\qquad (0 \le k \le n),
\end{equation}
where $m_k := 0$ for $k = -1$ and $k = n$.
Note that, in particular,
\[
\sum_{i=0}^{n} (-1)^{n-i} e_i = 0.
\]

\subsection{Cyclic descent extension and hook-multiplicities}\label{sec:hook_criterion}
	
Let $\A\subseteq S_n$ be Schur-positive with $\Q(\A) = \ch(\phi)$.
Denote
\[
m_\lambda 
:= \langle \phi, \chi^\lambda \rangle 
= \langle \Q(\A), s_\lambda \rangle
\qquad (\forall \lambda \vdash n).
\]
For $\lambda = (n-k,1^k)$ we use the abbreviation 
\[
m_k := m_{(n-k,1^k)} 
\qquad (0 \le k < n).
\]

The {\em hook-multiplicity generating function} is defined as 
\[
M_{\A}(x):=\sum_{k=0}^{n-1} m_k x^k.
\]

The function $M_{\A}(x)$, where $\A$ 
is a conjugacy class, 
was studied and applied to the enumeration of unimodal permutations with a given cycle type by Thibon~\cite{Thibon}. Indeed, 

\begin{observation}\label{lem:1}
For every $0\le k <n$
\[
m_k = \card{\{a\in \A:\ \Des(a)=[k]\}}.
\]
\end{observation}

\begin{proof}
For every $0 \le k < n$ there exists 
a unique standard Young tableau $T$ of size $n$ with
$\Des(T) = [k]$ (where $[0] := \varnothing$). 
The shape of $T$ is $(n-k,1^k)$.
Comparing the coefficients of ${\bf x}^{[k]}$ on both sides of Equation~\eqref{eq:l2} completes the proof.
\end{proof}

We now restate and prove Lemma~\ref{lem:21}.

\begin{lemma}\label{lem:2}
A Schur-positive set $\A \subseteq S_n$ carries a cyclic descent extension if and only if 
the hook-multiplicity generating function $M_\A(x)$ 
is divisible by $1+x$
and the quotient $M_\A(x)/(1+x)$ has non-negative coefficients; 
equivalently, if and only if
there exist non-negative integers $d_k$ $(0 \le k \le n-2)$ such that
\[
m_k = d_k + d_{k-1} \qquad (0 \le k \le n-1),
\]
where $d_k := 0$ for $k = -1$ and $k = n-1$.
\end{lemma}
	
\begin{proof}
If $\A$ carries
a cyclic descent extension
then, by Observation~\ref{lem:1} and the equivariance of $\cDes$, 
for every $0 \le k \le n-1$:
\begin{align*}
m_k
&= \card{\{a\in \A:\ \Des(a)=[k]\}} \\
&= \card{\{a\in \A:\ \cDes(a)=[k]\}}+\card{\{a\in \A:\ \cDes(a)=[k]\sqcup\{n\}\}} \\
&= \card{\{a\in \A:\ \cDes(a)=[k]\}}+\card{\{a\in \A:\ \cDes(a)=[k+1]\}}.
\end{align*}
The numbers
\[
d_k := \card{\{a\in \A:\ \cDes(a)=[k+1]\}}
\qquad (-1 \le k \le n-1)
\]
satisfy the required conditions,
which imply the corresponding properties of $M_\A(x)$.

For the opposite direction, assume that
there exist non-negative integers $d_k$ 
(with $d_{-1} = d_{n-1} = 0$)
such that $m_k=d_{k-1}+d_k$ for all $0\le k\le n-1$. 
By Pieri's rule~\cite[Theorem 7.15.7]{EC2}, 
\[
s_{(1^k) \oplus (n-k)}=s_{(1^k)} s_{(n-k)}=s_{(n-k+1,1^{k-1})}+s_{(n-k,1^k)}
\qquad (1 \le k \le n-1).
\]
Hence
\[
\sum_{k=1}^{n-1} d_{k-1} s_{(1^k)\oplus (n-k)}=
\sum_{k=1}^{n-1} d_{k-1} (s_{(n-k+1,1^{k-1})}+s_{(n-k,1^k)})
=\sum_{k=0}^{n-1} m_k s_{(n-k,1^k)}. 
\]
Since
\[
\Q(\A) 
= \sum_{\lambda \vdash n} \langle \Q(\A), s_\lambda \rangle s_\lambda 
= \sum_{\substack{\lambda\vdash n \\ \lambda \ \text{non-hook}}} m_\lambda s_\lambda +
\sum_{k=0}^{n-1} m_k s_{(n-k,1^k)}
\]
we obtain
\begin{equation}\label{eq:q}
\Q(\A) 
= \sum_{\substack{\lambda\vdash n \\ \lambda \ \text{non-hook}}} m_\lambda s_\lambda +
\sum_{k=1}^{n-1} d_{k-1} s_{(1^k)\oplus (n-k)}.
\end{equation}
By Corollary~\ref{cor:Schur-gf}, this is equivalent to
\[
\sum_{a\in \A} {\bf x}^{\Des(a)} = \sum_{\substack{\lambda\vdash n \\ \lambda\ \text{non-hook}}} m_\lambda 
\sum_{T\in \SYT(\lambda)} {\bf x}^{\Des(T)} + 	\sum_{k=1}^{n-1} d_{k-1}
\sum_{T \in \SYT((1^k)\oplus (n-k))} {\bf x}^{\Des(T)}.
\]
By	Theorem~\ref{conj1},  
the set $\SYT(\lambda)$ carries a cyclic descent extension if and only if $\lambda \vdash n$ is not a hook; 
and each of the sets $\SYT((1^k)\oplus (n-k))$ $(1 \le k \le n-1)$ carries a cyclic descent extension. 
Hence $\A$ also carries a cyclic descent extension, completing the proof. 

\end{proof}

\begin{corollary}
If a Schur-positive set $\A$ carries a cyclic descent extension then 
\begin{equation}\label{eq:main2}
\sum_{a\in \A} {\bf x}^{\cDes(a)} = \sum_{\substack{\lambda\vdash n \\ \lambda\ \text{non-hook}}} m_\lambda 
\sum_{T\in \SYT(\lambda)} {\bf x}^{\cDes(T)} + 	\sum_{k=1}^{n-1} d_{k-1}
\sum_{T\in \SYT((1^k)\oplus (n-k))}{\bf x}^{\cDes(T)},
\end{equation}
where $m_\lambda$ and $d_k$ are the non-negative integers defined above.
\end{corollary}

\begin{proof}
    By Corollary~\ref{cor:cDes-dist-Schur}
	together with Equation~\eqref{eq:q}, 
	the generating function for the corresponding cyclic descent set is uniquely determined
	and satisfies Equation~\eqref{eq:main2}. 
	\end{proof}

\subsection{Divisibility of the hook-multiplicity generating  function}
\label{sec:divisibility}	
	
Recall the notation 
\[
m_{k,\phi} :=
\langle \phi,\chi^{(n-k,1^k)} \rangle 
\qquad (0 \le k < n)
\]
for an $S_n$-character $\phi$.

\begin{lemma}\label{lem:div}
    For every $S_n$-character $\phi$, the hook-multiplicity generating function 
    \[
    M_\phi(x) 
    := \sum\limits_{k=0}^{n-1} m_{k,\phi} x^k
    \]
    is divisible by $1+x$ if and only if 
    the value of $\phi$ on an $n$-cycle is zero, i.e.,  $\phi_{(n)}=0$. 
\end{lemma}

\begin{proof}
    By~\cite[Lemma 4.10.3]{Sagan},  
    for every partition $\lambda\vdash n$ 
    \[
    \chi^\lambda_{(n)} =
    \begin{cases} 
        (-1)^k,&\text{if } \lambda = (n-k,1^k) \text{ for some } 0 \le k < n;\\
		0,&\text{otherwise.}
		\end{cases}
    \]
    Thus   
    \[
    \phi_{(n)}
    = \sum\limits_{\lambda\vdash n} \langle \phi,\chi^\lambda \rangle \chi^\lambda_{(n)}
    = \sum\limits_{k=0}^{n-1} \langle \phi,\chi^{(n-k,1^k)} \rangle \chi^{(n-k,1^k)}_{(n)}
    = \sum\limits_{k=0}^{n-1} m_{k,\phi} \cdot (-1)^k
    = M_\phi(-1),
    \]
    which equals zero if and only if $1+x$ divides $ M_\phi(x)$, 
    completing the proof.
\end{proof}

Letting $\phi = \psi^\lambda$, the higher Lie character indexed by a partition $\lambda$, reduces Proposition~\ref{prop:divisibility} to the following character evaluation. 


Recall the M\"obius function $\mu(n)$, the sum of all primitive (complex) $n$-th roots of $1$. If $n$ has a prime square divisor then $\mu(n)=0$; otherwise, $n$ is a product of $k$ distinct primes and $\mu(n)=(-1)^k$. 
The following lemma is equivalent to a combinatorial identity due to Garsia, as shown in  Proposition~\ref{prop:equiv_Garsia} below. We give here an independent direct algebraic proof.

\begin{lemma}\label{lem:char_eval_cycles}
For $\lambda \vdash n$
\[
    \psi^\lambda_{(n)} =
    \begin{cases} 
        \mu(r), &\text{if } \lambda = (r^s);\\
	    0, &\text{otherwise,}\\
    \end{cases}
\]
where $\mu$ is the M\"obius function.
\end{lemma}

\begin{proof}
%
Let $c$ be an $n$-cycle in $S_n$, and let  $Z_\lambda = Z_{S_n}(g)$ be the centralizer in $S_n$ of a specific element $g\in \C_\lambda$. 
An explicit formula for the induced character~\cite[(5.1)]{Isaacs} is
\[\psi^\lambda(c)=\omega^\lambda\uparrow_{Z_\lambda}^{S_n}(c)=\frac{1}{|Z_\lambda|}\sum_{\substack{x\in S_n\\x^{-1}cx\in Z_\lambda}}\omega^\lambda(x^{-1}cx).
\]
An $n$-cycle commutes only with its own powers. Thus, if $\lambda$ is {\em not} of the form $(r^s)$ for some $r$ and $s$, 
then there is no $n$-cycle in $Z_\lambda$;
equivalently, $x^{-1}cx \not\in Z_\lambda$ for every $x\in S_n$. 
It follows that, for such partitions $\lambda \vdash n$, $\psi^\lambda_{(n)}=0$.

Assume now that $\lambda= (r^s)$, and let let $g = g_1 g_2 \cdots g_s \in \C_\lambda$ be a fixed  product of $s$ disjoint $r$-cycles. 
The order of the centralizer 
$Z_\lambda = Z_{S_n}(g)$ is $s!r^s$. 
If $u \in Z_\lambda$ is an $n$-cycle ($n = rs$) then 
$g = u^k$ for some integer $k$ with $\gcd(k,n) = s$;
equivalently, $u^s = g^j$ for some $0 < j < r$ with $\gcd(j,r) = 1$.
Conversely, if $u \in S_n$ is an $n$-cycle satisfying 
$u^s = g^j$ for some $0 < j < r$ with $\gcd(j,r) = 1$, then $g$ is a power of $u$ and therefore $u \in Z_\lambda$.
Thus the number of $n$-cycles in $Z_\lambda$, namely the number of elements of $Z_\lambda \cap \C_{(n)}$, is $\varphi(r) (s-1)! r^{s-1}$, where $\varphi$ is Euler's totient function. 

Viewing $Z_\lambda$ as the group of $s \times s$ monomial (``generalized permutation'') matrices whose nonzero entries are complex $r$-th roots of unity, an element $u \in Z_\lambda \cap \C_{(n)}$ corresponds to a matrix whose underlying permutation is a full $s$-cycle and the product of its nonzero entries is a primitive $r$-th root of unity. 
This product is equal to $\omega^\lambda(u)$, so it is a primitive $r$-th root of unity.
For $u, v \in Z_\lambda \cap \C_{(n)}$, write $u \sim v$ if $v = u^i$ for some integer $i$ (necessarily coprime to $n$). This clearly defines an equivalence relation on $Z_\lambda \cap \C_{(n)}$. On each equivalence class, all primitive $r$-th roots of unity appear with the same frequency as values of $\omega^\lambda$. This property thus holds for all of $Z_\lambda \cap \C_{(n)}$, where this frequency is $(s-1)!r^{s-1}$. Denoting by $\xi$ any specific primitive $r$-th root of unity, the sum of all values of $\omega^\lambda$ on $Z_\lambda \cap \C_{(n)}$ is therefore
\[
\sum_{u\in Z_\lambda\cap\C_{(n)}} \omega^\lambda(u)
= (s-1)! r^{s-1} \sum_{j:(j,r)=1} \xi^j
= (s-1)! r^{s-1} \mu(r).
\]
Given any $c,u \in \C_{(n)}$, there are exactly $n=rs$ permutations $x \in S_n$ which satisfy $u=x^{-1}cx$.
Thus
\begin{align*}
\psi^\lambda(c)
&= \omega^\lambda\uparrow_{Z_\lambda}^{S_n}(c)
= \frac{1}{|Z_\lambda|}\sum_{\substack{x \in S_n \\ x^{-1}cx \in Z_\lambda}} \omega^\lambda(x^{-1}cx)
= \frac{n}{|Z_\lambda|}\sum_{u \in Z_\lambda\cap\C_{(n)}} \omega^\lambda(u) \\
&= \frac{n}{s!r^s} (s-1)! r^{s-1} \mu(r)
= \mu(r).
\end{align*}

\end{proof}

\begin{proof}[Proof of Proposition~\ref{prop:divisibility}.]
	By Lemma~\ref{lem:div}, 
	$1+x$ divides the hook-multiplicity generating function of the higher Lie character $\psi^\lambda$ 
    if and only if $\psi^\lambda_{(n)}=0$.
    Lemma~\ref{lem:char_eval_cycles} completes the proof.
\end{proof}

We deduce

\begin{corollary}\label{cor:only if}
Let $\lambda \vdash n$.
\begin{itemize}
\item[1.]    
If $\lambda = (r^s)$ for some square-free integer $r$ and positive integer $s$, then $1+x$ does not divide the hook-multiplicity generating function $M_\lambda(x)$, and the descent set map on the conjugacy class $\C_\lambda$ does not have a cyclic extension.
\item[2.] 
If $\lambda$ is not equal to $(r^s)$ for any square-free $r$, then $1+x$ divides 
$M_\lambda(x)$. 
In this case, the descent set map on 
$\C_\lambda$ has a cyclic extension if and only if 
the quotient $M_\lambda(x)/(1+x)$ has non-negative coefficients.
\end{itemize}
\end{corollary}

\begin{proof}
By the Gessel-Reutenauer theorem (Theorem~\ref{thm:GR22}), for every $\lambda\vdash n$ the conjugacy class $\C_\lambda$ 
is Schur-positive, with $\Q(\C_\lambda) = \ch(\psi^\lambda)$.
Combining this with Lemma~\ref{lem:2}  and  Proposition~\ref{prop:divisibility} completes the proof of both parts.
\end{proof}

In the following sections we will prove the non-negativity of the coefficients of the quotient $M_\lambda(x)/(1+x)$ for partitions (cycle types) which are not equal to $(r^s)$ for a square-free $r$: 
cycle types with more than one cycle length will be considered in Section~\ref{sec:distinct}, 
$n$-cycles will be considered in Section~\ref{sec:cycles},
and cycle types $\lambda = (r^s)$ with non square-free $r$ and $s>1$ will be considered in Section~\ref{sec:Witt} (this is the most difficult case).



\section{Non-negativity: the case of more than one cycle length}
\label{sec:distinct}
Consider, first, the case of a conjugacy class with more than one cycle length. This is the easiest case to handle.

\begin{proof}[Proof of Lemma~\ref{lem:distinct}.]
The centralizer $Z_\lambda$ of a permutation in $\C_\lambda$ is isomorphic, in this case, to the direct product $Z_\mu \times Z_\nu$. 
By Definition~\ref{def:higher}, 
$\omega^\lambda := \omega^{\mu} \otimes \omega^\nu$ and  
\begin{equation}\label{eq:distinct}
\psi^\lambda 
= \omega^\lambda\uparrow_{Z_\lambda}^{S_n}
= \left(\omega^{\mu}\uparrow_{Z_{\mu}}^{S_{|\mu|}} \otimes \, 
		\omega^{\nu}\uparrow_{Z_{\nu}}^{S_{|\nu|}}\right)
		\uparrow_{S_{|\mu|}\times S_{|\nu|}}^{S_n}.		    
\end{equation}
By the Littlewood-Richardson rule~\cite[Theorem A1.3.3]{EC2}, the outer product  of two irreducible characters 
$\left( \chi^\alpha\otimes\chi^\beta\right)\uparrow_{S_m\times S_{n-m}}^{S_n}$ 
contains irreducible representations indexed by hooks if and only if both $\alpha$ and $\beta$ are hooks; in the latter case,
\[
\langle \left( \chi^{(m-i,1^i)} \otimes \chi^{(n-m-j,1^j)}\right)\uparrow_{S_m \times S_{n-m}}^{S_n}, \chi^{(n-k,1^k)}\rangle = \begin{cases}
        1, & \ k\in \{i+j, i+j+1\};\\
        0, & \ \rm{otherwise.}
\end{cases}
\]
Therefore
\[
M_\lambda(x) = (1+x) M_\mu(x) M_\nu(x),
\]
as claimed.
\end{proof}	

We deduce 

\begin{corollary}
If $\lambda$ is a partition with more than one cycle length then
$M_\lambda(x)$ is divisible by $1+x$ and the quotient has non-negative coefficients.
\end{corollary}

\begin{remark}\label{rem:distinct} 
In this case, the existence of a cyclic descent extension
may be proved directly as follows.
By Equation~\eqref{eq:distinct}, $\psi^\lambda$ is a sum of characters indexed by disconnected shapes. 
Thus, by 
Corollary~\ref{cor:Schur-gf}, the distribution of the descent set over $\C_\lambda$ is equal to a sum of distributions over the sets of SYT of various disconnected shapes.
By Theorem~\ref{conj1}, each of these sets carries a cyclic descent extension, hence so does $\C_\lambda$.
\end{remark}

\section{Non-negativity: the single cycle case}
\label{sec:cycles}

Consider now the case of a conjugacy class with a single cycle.
By Corollary~\ref{cor:only if}, if $r$ is not square-free then $1+x$ divides $M_{(r)}(x)$. 
The main result of this section is

\begin{proposition}\label{prop:cycles_nonnegativity}
    If $r$ is not square-free then the coefficients of $M_{(r)}(x)/(1+x)$ are non-negative
\end{proposition}
	 
	
		
It follows from Lemma~\ref{t:unimodal} below that, in order to prove Proposition~\ref{prop:cycles_nonnegativity}, it suffices to show the unimodality (to be defined) of $M_{(r)}(x)$. This is the content of the following statement.

		 
\begin{proposition}\label{prop:s=1_unimodal}
    For every positive integer $r$, the sequence $m_{0,(r)}, m_{1,(r)},\ldots, m_{r-1,(r)}$ is unimodal.
    The largest element is one of the middle ones, namely $m_{i,(r)}$ for $i = (r-1)/2$ if $r$ is odd, and either $i = (r-2)/2$ or $i = r/2$, or both, if $r$ is even.
\end{proposition}

In Subsection~\ref{sec:Witt-subsec} we use a variant of the Witt transform to produce explicit formulas for the coefficients $m_{j,(r)}$ (Lemma~\ref{t:m_formula}).
Then, in Subsection~\ref{sec:unimodality}, we prove their unimodality.




\subsection{A variant of the Witt transform}
\label{sec:Witt-subsec}
 
In this subsection we present a variant of the Witt transform, which will be used to prove non-negativity in Sections ~\ref{sec:unimodality} and~\ref{sec:Witt}.
 
	


    
Denote by $(i,j)$ the greatest common divisor of two integers $i,j$. Recall the arithmetical M\"obius function $\mu$. 
    
\begin{definition}\label{def:f}
For a positive integer $r$ define 
\[
    f_j(r) 
    := \frac{1}{r} \sum_{d|(r,j)} \mu(d)(-1)^{(d+1)j/d} \binom{r/d}{j/d} 
	\qquad (0\le j\le r). 
\]
\end{definition}
   
\begin{observation}\label{obs:f_values}
    By definition, 
    \[
        f_1(r) = f_{r-1}(r) = 1
        \qquad (r \ge 1).
    \]
    Also, the fundamental property
    \[
        \sum_{d|r} \mu(d) = 
        \begin{cases}
            1, &\text{if } r=1; \\
            0, &\text{if } r>1,
        \end{cases}
    \]
    and some case analysis ($r$ odd, or $r \equiv 0 \pmod 4$, or $r \equiv 2 \pmod 4$) imply that
    \[
        f_0(r) = 
        \begin{cases} 
            1,&\text{ if }  r=1;\\
            0,&\text{ if } r>1\\
        \end{cases}
        \qquad \text{and} \qquad
        f_r(r) = 
        \begin{cases} 
            1,&\text{ if }  r=1,2;\\
            0,&\text{ if } r>2.
        \end{cases}
    \]
\end{observation}
  

    
For the higher Lie character $\psi^{(r)}$, 
simplify slightly the notations in Equation~\eqref{e_k_interpretation1}:
\[
    m_{j,(r)}
    := \langle \psi^{(r)},\chi^{(r-j,1^j)} \rangle 
    \quad (0 \le j \le r-1)
    \quad \text{and}\quad	
    e_{j,(r)} 
    :=\langle \psi^{(r)},\chi^{(1^j)\oplus(r-j)} \rangle
    \quad (0 \le j \le r).
\]

\begin{proposition}\label{prop:e_i_formuli1}
    For every $0\le j\le r$ 
    \[
    e_{j,(r)} = f_j(r).
    \]
    In particular, $f_j(r)$ is a non-negative integer.
\end{proposition}

\begin{remark}\label{rem:e=s}
    Proposition~\ref{prop:e_i_formuli1} will not be proved here, since it is the special case $s=1$ of Theorem~\ref{prop:e_i_formuli} below. 
    It also follows from a well known result of Kra\'skiewicz and Weyman~\cite{KraskiewiczWeyman} (Lemma~\ref{lem:Kra-Wey} below).
    A symmetric functions proof which applies~\cite[Lemma 2.7]{Sundaram_Adv} was presented by Sheila Sundaram~\cite{Sundaram_personal}. 
    Another proof follows from~\cite[Theorem 3.1]{ET}. 
   	See Subsection~\ref{remarks_on_s=1} below for a discussion.
\end{remark}
 

 
\begin{definition}\label{def:F_polynomial}
    For a fixed positive integer $r$, collect the multiplicities $f_j(r)$ into a polynomial
    \[
	    F_r(x) := f_0(r) + f_1(r)x + f_2(r)x^2+ \ldots + f_r(r)x^r.
    \]
\end{definition}

Equation~\eqref{eq:alternating} and Proposition~\ref{prop:e_i_formuli1} imply
\begin{corollary}\label{t:F_and_M}
    \[
        F_r(x) = (1+x) M_{(r)}(x).
    \]
\end{corollary}

\begin{observation}\label{t:F}
\[
F_r(x) = \frac{1}{r} \sum_{d|r} \mu(d) (1-(-x)^d)^{r/d}.
\]
\end{observation}
\begin{proof}
Use Definition~\ref{def:f}, and write $j = kd$ if $d | (r,j)$. Then
\begin{align*}
F_r(x) 
&= \sum_{j=0}^{r} x^j \sum_{d|(r,j)} \frac{\mu(d) (-1)^{(d+1)j/d}}{r} \binom{r/d}{j/d} 
= \sum_{d|r} \frac{\mu(d)}{r} \sum_{k=0}^{r/d} (-1)^{(d+1)k} \binom{r/d}{k}x^{k d} \\
&= \sum_{d|r} \frac{\mu(d)}{r} (1-(-x)^d)^{r/d}.
\end{align*}

\end{proof}
	
\begin{remark}
Recall from \cite{Moree} that the $r$-th Witt transform of a polynomial $p(x)$ is defined by 
\[
\mathcal{W}_p^{(r)}(x)=\frac{1}{r}\sum_{d|r}\mu(d)p(x^d)^{r/d}.
\] 
In our case, put $p(x)=1-x$ to get $F_r(x) = \mathcal{W}_p^{(r)}(-x)$. 
The proof of Theorem~4 and Lemma~1 in~\cite{Moree} could have been used to prove that the coefficients of $F_r(x)$ are non-negative integers. This non-obvious property of the numbers $f_j(r)$ also follows, of course, from their interpretation in Proposition~\ref{prop:e_i_formuli1} as inner products of two characters.
What we really need, in Proposition~\ref{prop:cycles_nonnegativity}, is the nonnegativity of the coefficients of $F_r(x)/(1+x)^2$.
\end{remark}

We now produce an explicit formula for each coefficient $m_{j,(r)}$. For a combinatorial interpretation of these numbers, see Lemma~\ref{lem:Kra-Wey} below.

\begin{lemma}\label{t:m_formula}
    For a positive integer $r$,
    \[
		m_{j,(r)}
		= \frac{1}{r} \sum_{d|r} \mu(d) \binom{r/d - 1}{\lfloor j/d \rfloor}  (-1)^{j+\lfloor j/d \rfloor} 
		\qquad (0 \le j \le r-1).
    \]
\end{lemma}
\begin{proof}
	By Corollary~\ref{t:F_and_M},
	\[
	    (1+x) \sum_{j=0}^{r-1} m_{j,(r)}x^j=F_r(x). 
	\]
	
	Using Definition~\ref{def:f} and Observation~\ref{t:F},
	we can write
	\[
	    \sum_{j=0}^{r-1} m_{j,(r)} x^j 
	    = \frac{1}{r(1+x)} \sum_{d|r} \mu(d) (1-(-x)^d)^{r/d}.
    \]
	Using
	\[
		\frac{(1-(-x)^d)^{r/d}}{1+x} 
		= (1-(-x)^d)^{r/d-1} (1 - x + x^2 - \ldots +(-x)^{d-1})
	\]
	and comparing coefficients of $x^j$, where $j = dk + \ell$ with $0 \le \ell \le d-1$, we get
	\[
		r m_{j,(r)}
		= \sum_{d|r} \mu(d) \binom{r/d - 1}{k}  (-1)^{(d+1)k + \ell}
		= \sum_{d|r} \mu(d) \binom{r/d - 1}{\lfloor j/d \rfloor}  (-1)^{j+\lfloor j/d \rfloor}.
	\]
\end{proof}

\subsection{Unimodality}
\label{sec:unimodality}
	
A sequence $(a_0,\ldots,a_n)$ of real numbers is called {\em unimodal} if there exists an index $0 \le i_0 \le n$ such that the sequence is weakly increasing up to position $i_0$ and weakly decreasing afterwards:
$a_0 \le a_1 \le \ldots \le a_{i_0} \ge \ldots a_{n-1} \ge a_n$.
	
\begin{lemma}\label{t:unimodal}
Let $a(x) = a_0 + a_1 x + \ldots + a_n x^n$ be a polynomial with real, nonnegative and unimodal coefficients.
Assume that $1+x$ divides $a(x)$, and let $b(x) := a(x)/(1+x)$.
Then the coefficients of $b(x)$ are nonnegative.
\end{lemma}%
	
\begin{proof}
Let $b(x) = b_0 + \ldots + b_{n-1} x^{n-1}$.
Then $a_0 = b_0$, $a_n = b_{n-1}$, and
\begin{equation}\label{eq:unimodal1}
a_i = b_{i-1} + b_i 
\qquad (1 \le i \le n-1)
\end{equation}
Of course, divisibility of $a(x)$ by $1+x$ implies that
\[
\sum_{i=0}^{n} (-1)^i a_i = a(-1) = 0.
\]
Inverting~\eqref{eq:unimodal1} we get
\begin{equation}\label{eq:unimodal2}
b_i = \sum_{j=0}^i (-1)^{i-j} a_j
\qquad (0 \le i \le n-1)
\end{equation}
and, similarly,
\begin{equation}\label{eq:unimodal3}
b_i = \sum_{j=i+1}^n (-1)^{j-i-1} a_j
\qquad (0 \le i \le n-1).
\end{equation}
By assumption, the sequence $(a_0, \ldots, a_n)$ is nonnegative and unimodal, namely: 
there exists an index $0 \le i_0 \le n$ such that
\[
0 \le a_0 \le \ldots \le a_{i_0} \ge \ldots \ge a_n \ge 0.
\]
It follows from~\eqref{eq:unimodal2} that, for odd indices $0 \le 2i+1 \le i_0$,
\[
b_{2i+1} = (a_{2i+1} - a_{2i}) + \ldots + (a_1 - a_0) \ge 0
\]
and, for even indices $0 \le 2i \le i_0$,
\[
b_{2i} = (a_{2i} - a_{2i-1}) + \ldots + (a_2 - a_1) + a_0 \ge 0.
\]
By~\eqref{eq:unimodal3}, a similar argument holds for 
indices greater or equal to $i_0$, and the proof is complete.
\end{proof}
	

Lemma~\ref{t:unimodal} shows that non-negativity of a sequence can be proved by showing unimodality of a related sequence. In particular, Proposition~\ref{prop:cycles_nonnegativity} would follow once we show the unimodality of the polynomial $M_{(n)}(x)$.
In order to do that, we need the following technical lemma.

\begin{lemma}\label{lem:binomial_approx}
    Assume that $r > 7$ and $1 < j < r/2$;
    if $j = (r-1)/2$ assume also that $r > 11$.
    Let $d > 1$ be a divisor of $r$, and denote
	\[
	A_{r,j,d} 
	:= \frac{\displaystyle\binom{r/d - 1} {\lfloor j/d \rfloor}} {\displaystyle\binom{r-1}{j}}.
	\]
	Then
	\[
	\frac{(r-1)(r-j)}{r-2j} A_{r,j,d} \le 
	\begin{cases}
	        1, &\text{if } d>2; \\
	        3/2, &\text{if } d=2.
	\end{cases}
	\]
\end{lemma}
	
	
\begin{proof}
    Write
	\[
	\binom{r-1}{j}
	= \frac{(r-1)(r-2) \cdots (r-j)}{j!}
	= \prod_{1 \le i \le j} \frac{r-i}{i}.
	\]
	Let $\ell := \lfloor j/d \rfloor$. Then $\ell d$ is the largest multiple of $d$ not exceeding $j$, hence
	\[
	\binom{r/d - 1}{\lfloor j/d \rfloor} 
	=\binom{r/d - 1}{\ell}
	= \frac{(r-d)(r-2d)\cdots (r-\ell d)}{d \cdot 2d \cdots \ell d}
	= \prod_{1 \le i \le j,\, d \mid i} \frac{r-i}{i},
	\]
	The quotient $A_{r,j,d}$ can therefore be written in the form
	\begin{equation*}
		A_{r,j,d} 
		= \prod_{1 \le i \le j,\, d \nmid i} \frac{i}{r-i}.
	\end{equation*}
	By assumption $j < r/2$, thus $i/(r-i) < 1$ for all $1 \le i \le j$. It follows that $A_{r,j,d}$ is a decreasing function of $j$, with $A_{r,1,d} = 1/(r-1)$.
	%
	%
    
    For $d > 2$ and $2 \le j \le (r-2)/2$,
    \[
    A_{r,j,d} \le A_{r,2,d} = \frac{2}{(r-1)(r-2)}
    \le \frac{r-2j}{(r-1)(r-j)},
    \]
    where the last inequality, equivalent to $2(r-j) \le (r-2j)(r-2)$, follows from $2 \le r-2j$ and $r-j \le r-2$. 
	
	Similarly, for $d=2$ and $3 \le j \le (r-2)/2$,
	\[
	A_{r,j,2} \le A_{r,3,2} = \frac{3}{(r-1)(r-3)} 
	\le \frac{3(r-2j)}{2(r-1)(r-j)},
	\]
	where the last inequality, equivalent to $2(r-j) \le (r-2j)(r-3)$, follows from $2 \le r-2j$ and $r-j \le r-3$.

    For $d = 2$ and $j = 2$,
    \[
    A_{r,2,2} = \frac{1}{r-1} \le \frac{3(r-4)}{2(r-1)(r-2)},
    \]
    where the inequality, equivalent to $2(r-2) \le 3(r-4)$, follows from $r \ge 8$.
    
		
		
    
    \smallskip
    
	It remains to consider the case $j = (r-1)/2$, namely $r = 2j+1$, for $d \ge 2$.
	Note that in this case we assumed that $r > 11$, namely $j > 5$.
	
	Assume first that $d > 2$. Then 
	\[
    A_{r,j,d} \le A_{r,6,d} \le A_{r,4,d}
    = \prod_{1 \le i \le 4,\, d \nmid i} \frac{i}{r-i}
    = \frac{a_d}{\displaystyle\binom{r-1}{4}},
    \]
    where
    \[
    a_d =
    \begin{cases}
        1, &\text{if } d > 4; \\
        (r-d)/d, &\text{if } d = 3,4.
    \end{cases}
    \]
    Clearly	
	\[
	1 < \frac{r-4}{4} < \frac{r-3}{3} 
	\]
	and therefore 
	\[
	A_{r,j,d} \le A_{r,4,d} \le A_{r,4,3} 
	= \frac{8}{(r-1)(r-2)(r-4)}
	\le \frac{r-2j}{(r-1)(r-j)},
	\]
	where the last inequality, equivalent (since $r = 2j+1$) to $8(j+1) \le (2j-1)(2j-3)$ and to $4j^2 - 16j \ge 5$, holds for $j \ge 5$.
	
	Finally, assume that $r = 2j+1$ and $d = 2$.
	Then 
	\[
	A_{r,j,2} \le A_{r,6,2} = A_{r,5,2} 
	= \frac{15}{(r-1)(r-3)(r-5)}
	\le \frac{3(r-2j)}{2(r-1)(r-j)},
	\]
	where the last inequality, equivalent (since $r = 2j+1$) to $5(j+1) \le 2(j-1)(j-2)$ and to $2j^2 - 11j \ge 1$, holds for $j \ge 6$.
	This completes the proof.

\end{proof}
	
\begin{proof}[Proof of Proposition~\ref{prop:s=1_unimodal}.]
    We need to show that $m_{0,(r)} \le m_{1,(r)} \le \ldots \le m_{\lfloor (r-1)/2 \rfloor}$ and     $m_{r-1,(r)} \le m_{r-2,(r)} \le \ldots \le m_{\lceil (r-1)/2 \rceil}$ for any positive integer $r$.
    
	For $1 \le r \le 7$, computing the polynomial $M_r(x) := F_r(x)/(1+x)$ explicitly, using Observation~\ref{t:F}, gives
	\begin{align*}
		M_1(x) &= 1; \qquad 
		M_2(x) = M_3(x) = x; \qquad
		M_4(x) = x+x^2; \\
		M_5(x) &= x+x^2+x^3; \qquad
		M_6(x) = x+2x^2+x^3+x^4; \\
		M_7(x) &= x+2x^2+3x^3+2x^4+x^5.
	\end{align*}
	The claim clearly holds in these cases.
	Assume from now on that $r > 7$.
	
    Informally, the explicit formula for $m_{j,(r)}$ in Lemma~\ref{t:m_formula} has a dominant term corresponding to $d=1$, i.e.,
    $rm_{j,(r)}$ is approximately equal to $\binom{r - 1}{j}$. 
    We will show that this approximation is good enough to make the sequence $m_{0,(r)}, \ldots, m_{r-1,(r)}$ unimodal, like the sequence of binomial coefficients.
    Note that, unlike the binomial coefficients, this sequence is not always palindromic; see Proposition~\ref{conj:r2} below.
    

	\medskip
	
	We first show that $m_{j-1,(r)} \le m_{j,(r)}$ for $1 \le j < r/2$. Recall that we assume $r > 7$.
	
	For $j=1$, Lemma~\ref{t:m_formula} shows that, 
	for $r > 1$, $m_{0,(r)} = 0 < 1 = m_{1,(r)}$. 
	
    Assume now that  $1 < j < r/2$.	
    Clearly, for these values of $j$ and any divisor $d$ of $r$,
	\[
	\binom{r/d - 1}{\lfloor j/d \rfloor} 
	\ge \binom{r/d - 1}{\lfloor (j-1)/d \rfloor}.
	\]
	We conclude, by Lemma~\ref{t:m_formula}, that
	\[
	    rm_{j,(r)} - rm_{j-1,(r)} 
	    \ge \left[ \binom{r-1}{j} - \binom{r-1}{j-1} \right] - 2\sum_{d|r,\,d>1} \binom{r/d - 1}{\lfloor j/d \rfloor}. 
	\]
	Since
		\[
	    \binom{r-1}{j} - \binom{r-1}{j-1}
	    = \binom{r-1}{j} \left(1 - \frac{j}{r-j}\right)
	    = \binom{r-1}{j} \frac{r-2j}{r-j},
	\]
    using the notation of Lemma~\ref{lem:binomial_approx} we get
	\[
	\frac{rm_{j,(r)} - rm_{j-1,(r)}} {\binom{r-1}{j}\frac{r-2j}{r-j}}
	\ge 
	        1 - 2\frac{r-j}{r-2j} \sum_{d|r,\,d>1} A_{r,j,d}.
	\]
	Let $d(r)$ denote the number of divisors of $r$. 
	For odd $r>7$ (unless $j = (r-1)/2$ and $r \in \{9,11\}$), Lemma~\ref{lem:binomial_approx} implies that
	\begin{align*}
		\frac{rm_{j,(r)} - rm_{j-1,(r)}} {\binom{r-1}{j}\frac{r-2j}{r-j}} 
		\ge 1 - \sum_{d|r,\,d>2} \frac{2}{r-1} 
		= 1 - \frac{2d(r)-2}{r-1}.
	\end{align*}
	For even $r>7$, Lemma~\ref{lem:binomial_approx} implies that
	\begin{align*}
		\frac{rm_{j,(r)} - rm_{j-1,(r)}} {\binom{r-1}{j}\frac{r-2j}{r-j}} 
		\ge 1 - \frac{3}{r-1} - \sum_{d|r,\,d>2} \frac{2}{r-1}
		= 1 - \frac{2d(r)-1}{r-1}.
	\end{align*}
	%
	%
    We clearly have $2d(r) \le r$ for $r>7$, and therefore $m_{j-1,(r)} \le m_{j,(r)}$ in both cases.
	
	In the remaining cases, namely $j = (r-1)/2$ and $r \in \{9,11\}$, we can compute directly using 	Lemma~\ref{t:m_formula}.
	For $r=9$ and $j=4$ the divisors are $d=1,3,9$, but $\mu(9) = 0$. Thus
	\[
		9m_{4,(9)}-9m_{3,(9)} 
		= \left[ \binom{8}{4} + \binom{2}{1} \right]
		- \left[ \binom{8}{3} - \binom{2}{1} \right]
		= 72 - 54 > 0.
	\]
	For $r=11$ and $j=5$ the divisors are $d=1,11$. Thus
	\[
		11m_{5,(11)}-11m_{4,(11)} 
		= \left[ \binom{10}{5} + \binom{0}{0} \right]
		- \left[ \binom{10}{4} - \binom{0}{0} \right]
		= 253 - 209 > 0.
	\]
    
	So far we have proven that $m_{0,(r)} \le m_{1,(r)} \le \ldots \le m_{\lfloor (r-1)/2 \rfloor, (r)}$ for $r > 7$.
	
	The remaining inequalities, $m_{r-1,(r)} \le m_{r-2,(r)} \le \ldots \le m_{\lceil (r-1)/2 \rceil, (r)}$, can be written as $m_{r - 1 - (j-1),(r)} \le m_{r - 1 - j,(r)}$ for $1 \le j < r/2$.
	By Lemma~\ref{t:m_formula},
	\[
		m_{r-1-j,(r)}
		= \frac{1}{r} \sum_{d|r} \mu(d) \binom{r/d - 1}{\lfloor (r-1-j)/d \rfloor}  (-1)^{r-1-j+\lfloor (r-1-j)/d \rfloor} 
		\qquad (0 \le j \le r-1).
    \]
    For a divisor $d$ of $r$, if $j = kd + \ell$ with $0 \le \ell \le d-1$ 
    then $r-1-j = (r/d-k-1)d + (d-1-\ell)$ with $0 \le d-1-\ell \le d-1$, 
    so that
    \[
    \lfloor j/d \rfloor + \lfloor (r-1-j)/d \rfloor
    = k + (r/d-k-1)) = r/d-1.
    \]
    It follows that
	\begin{align*}
		m_{r-1-j,(r)}
		&= \frac{1}{r} \sum_{d|r} \mu(d) \binom{r/d - 1}{\lfloor j/d \rfloor}  (-1)^{r-1-j+r/d-1-\lfloor j/d \rfloor} \\
		&= \frac{1}{r} \sum_{d|r} \mu(d) \binom{r/d - 1}{\lfloor j/d \rfloor}  (-1)^{r+r/d-j-\lfloor j/d \rfloor} 
		\qquad (0 \le j \le r-1).
    \end{align*}
    This is exactly the formula for $m_{j,(r)}$ except for the signs of the summands, which differ (for each $d|r$) by the factor $(-1)^{r+r/d}$.
    These signs do not play any role in the proof above that $m_{j-1,(r)} \le m_{j,(r)}$ for $1 \le j < r/2$, which therefore also shows, mutatis mutandis, that $m_{r-1-(j-1),(r)} \le m_{r-1-j,(r)}$ for $1 \le j < r/2$ --- except possibly the explicit confirmation when $j = (r-1)/2$ and $r \in \{9,11\}$. 
    In these cases $r$ is odd, and therefore $(-1)^{r+r/d} = 1$ for any divisor $d$ of $r$.
    This implies that indeed 
	\[
		m_{4,(9)}-m_{5,(9)} 
		= m_{4,(9)}-m_{3,(9)}
		> 0
	\]
	and
	\[
		m_{5,(11)}-m_{6,(11)} 
		= m_{5,(11)}-m_{4,(11)} 
		> 0.
	\]
\end{proof}

\begin{remark}
	We conjecture that Proposition~\ref{prop:s=1_unimodal} remains true for an arbitrary partition $\mu \vdash n$, in particular for every partition $(r^s) \vdash n$ with any	$s\ge 1$; see Conjecture~\ref{conj:unimodal}.
\end{remark}


\begin{proof}[Proof of Proposition~\ref{prop:cycles_nonnegativity}.]
    Combine Proposition~\ref{prop:s=1_unimodal} with   Lemma~\ref{t:unimodal}.
\end{proof}

We conclude 

	 \begin{corollary}\label{prop:CDE_cycles}
	 The conjugacy class of $n$-cycles $\C_{(n)}$ carries a cyclic descent extension if and only if $n$ is not square-free.
	 \end{corollary}
	 
	 \begin{proof}
Combine 
Lemma~\ref{lem:2} with 
Corollary~\ref{cor:only if}.2 and Proposition~\ref{prop:cycles_nonnegativity} 
	 \end{proof}

\section{Non-negativity: the case of cycle type $(r^s)$}
\label{sec:Witt}

In this section we consider the case $\lambda = (r^s)$. 
We fix $r$, while $s$ and hence $n=rs$ vary. 
The arguments below also work for the trivial case $r=1$.
		
As in the previous section, instead of the hook multiplicities 
\[
    m_{i,(r^s)} 
    = \langle \psi^{(r^s)}, \chi^{(n-i,1^i)} \rangle
\]
we prefer to work with their consecutive sums,
\[
    e_{i,(r^s)} := m_{i,(r^s)} + m_{i-1,(r^s)}.
\]
		
Here is the structure of the current section.
In Subsection~\ref{sec:e} 
we obtain an explicit description of $e_{i,(r^s)}$ (Theorem~\ref{prop:e_i_formuli}). The proof involves a detailed computation of character values and inner products of characters.
In Subsection~\ref{sec:prod} we 
transform this description into
a product formula (Corollary~\ref{power_series_form})
for the formal power series 
\[
E_r(x,y) 
:= \sum_{i,s \ge 0} e_{i,(r^s)} x^i y^s 
= 1 + (1 + x) M_r(x,y),  
\]
where 
\[
M_r(x,y) 
:= \sum_{\substack{i \ge 0 \\ s \ge 1}} m_{i,(r^s)} x^i y^s.
\]
The product formula is a substantial merit of working with $e_{i,(r^s)}$, and it facilitates the extension of the case $s=1$ to $s>1$. 
This is done in Subsection~\ref{sec:nn}, where the result for $s=1$ is used to obtain the general case.
	
		

		

\subsection{Formulas for inner products}\label{sec:e}
	
In this subsection we obtain explicit formulas for the inner products 
\[
e_{k,(r^s)} := \langle \psi^{(r^s)}, \chi^{(1^k) \oplus (n-k)} \rangle
\quad (0 \le k \le n).
\]
Recall that, by Equation~\eqref{eq:alternating}, 
	%
	%
the sequences $\{m_k\}_{k=0}^{n-1}$ and $\{e_k\}_{k=0}^{n}$ 
determine each other, via the relations
\[
e_k = m_k+m_{k-1}
\text{ and }
m_k = \sum_{i=0}^k (-1)^{k-i} e_i
\qquad (0 \le k \le n),
\]
where $m_k := 0$ for $k = -1$ and $k = n$.
Nota bene, these multiplicities depend on $r$ and $s$
but this dependence is suppressed in the notation.

\begin{definition}\label{def:partition}
For given non-negative integers $i$, $r$ and $s$, let
\[
    P_{r,s}(i) 
    := \left\{ \gamma=(\gamma_1,\ldots,\gamma_s) :
    \sum_{\ell = 1}^{s} \gamma_\ell = i,\, r \ge \gamma_1 \ge \gamma_2 \ge \ldots \ge \gamma_s \ge 0 \right\}
\]
denote the set of all partitions of $i$ into at most $s$ parts, each of size at most $r$. 
Denote the multiplicity of $j$ in $\gamma \in P_{r,s}(i)$ by 
$k_j(\gamma) := |\{1 \leq \ell \leq s \mid \gamma_\ell=j\}|$.
\end{definition}
	
\begin{figure}
    \ytableausetup{boxsize=1.3em}
    \ydiagram[*(gray)]{5,3,3,2} * {6,6,6,6,6}
    \caption{The partition $\gamma = (5,3,3,2,0) \in P_{6,5}(13)$
	}
	\label{fig:subpart}
\end{figure}
	
\begin{example}
Let $r=6$, $s=5$ and $i=13$. 
Then $\gamma=(5,3,3,2,0)\in P_{6,5}(13)$ is a partition of $13$ with at most $5$ parts, each of size at most $6$; 
see Figure~\ref{fig:subpart}. 
The multiplicities of the parts are 
$k_0(\gamma) = 1$, 
$k_1(\gamma) = 0$, 
$k_2(\gamma) = 1$, 
$k_3(\gamma) = 2$, 
$k_4(\gamma) = 0$, 
$k_5(\gamma) = 1$, 
and $k_6(\gamma) = 0$.
\end{example}

Recall $f_j(r)$ from Definition~\ref{def:f}. 
The main result of this subsection is the following formula.

\begin{theorem}\label{prop:e_i_formuli}
For every $s\geq 1$ and $i\geq 0$ we have
\begin{align*}
e_{i,(r^s)}
= \langle \psi^{(r^s)}, \chi^{(1^i) \oplus (rs-i)}
\rangle 
&= \sum_{\gamma\in P_{r,s}(i)} \prod_{j \ge 0}
(-1)^{(j+1)k_j(\gamma)} \binom{(-1)^{j+1} f_j(r)}{k_j(\gamma)} \\
&= \sum_{\substack{k_0, \dots, k_r \ge 0 \\ \sum_j k_j = s \\ \sum_j jk_j = i}} \prod_{j = 0}^{r}
(-1)^{(j+1)k_j} \binom{(-1)^{j+1} f_j(r)}{k_j}.
\end{align*}
		
In particular, for $s=1$ we have $e_{i,(r)} = f_i(r)$.
\end{theorem}

\begin{remark}\label{rem:when_factor_is_zero}
The special case $s=1$ was stated (but not proved) in  Proposition~\ref{prop:e_i_formuli1}.
The result in that case is not new, as noted in Remark~\ref{rem:e=s}. 
This case
shows that $f_i(r) = e_{i,(r)}$ is an inner product of two characters, and is therefore always a non-negative integer. The factor
\[
(-1)^{(j+1)k_j} \binom{(-1)^{j+1} f_j}{k_j}
= \begin{cases}
\binom{f_j}{k_j}, &\text{ if } j \text{ is odd;}\\
\binom{f_j+k_j-1}{k_j}, &\text{ if } j \text{ is even}\\
\end{cases}
\]
is therefore also a non-negative integer,
and is zero if and only if 
either $j$ is odd and $k_j > f_j$,
or $j$ is even and $k_j > 0 = f_j$.
If $k_j = 0$, this factor is equal to $1$ and may be ignored.
	
	
	
	
\end{remark}
    
In order to prove Theorem~\ref{prop:e_i_formuli} we need a formula for a certain inner product of characters (Lemma~\ref{lem:char_values_and_inner_product}).

First recall some notations from Definition~\ref{def:higher}:
the centralizer $Z_{(r^s)} \cong \ZZ_r \wr S_s$ of an element of cycle type $(r^s)$,
the linear character $\omega^{(r^s)}$ on $Z_{(r^s)}$, 
and the higher Lie character
$\psi^{(r^s)} := \omega^{(r^s)}\uparrow_{Z_{(r^s)}}^{S_n}$. 

Embed $Z_{(r^s)} \cong \ZZ_r \wr S_s$ into 
$K_{r,s} \cong S_r \wr S_s \le S_n$, 
where $\ZZ_r \le S_r$ is generated by a full cycle.
Denote 
\[
\phi_{r,s} := \omega^{(r^s)}\uparrow_{Z_{(r^s)}}^{K_{r,s}},
\]
so that $\psi^{(r^s)} = \phi_{r,s}\uparrow_{K_{r,s}}^{S_n}$. 
\begin{observation}\label{obs:phi_product} If $s=s_1+s_2$ then $Z_{(r^{s_1})}\times Z_{(r^{s_2})}\leq Z_{(r^s)}$, $K_{r,s_1}\times K_{r,s_2}\leq K_{r,s}$ and also for the characters
\[\omega^{(r^s)}=\omega^{(r^{s_1})}\otimes\omega^{(r^{s_2})}\text{ and }  \phi_{r,s}\downarrow^{K_{r,s}}_{K_{r,s_1}\times K_{r,s_2}}=\phi_{r,s_1}\otimes\phi_{r,s_2}.
\]
\end{observation}

In the lemma below we express the multiplicity of a certain linear character in a restriction of $\phi_{r,s}$.
This expression will be used, in the proof of Theorem~\ref{prop:e_i_formuli}, to compute $e_{i,(r^s)}$.
	
For every $0\leq j\leq r$, let 
$R_{r,j} := S_j \times S_{r-j} \le S_r$, in the natural embedding. Then
\[
R_{r,j} \wr S_s
= K_{r,s} \cap \left( S_{js} \times S_{(r-j)s} \right).
\]
Denote by $1_{S_n}$ the trivial character and by $\varepsilon_{S_n}$ the sign character of $S_n$, so that
\[
\chi^{(1^k)\oplus(n-k)}=(\varepsilon_{S_{k}}\times 1_{S_{n-k}})\uparrow_{S_k \times S_{n-k}}^{S_n}.
\]
Define
\[
\nu_{r,j,s} := (\varepsilon_{S_{js}} \times 1_{S_{(r-j)s}}) \downarrow_{R_{r,j}\wr S_s}^{S_{js} \times S_{(r-j)s}}.
\]
It is a linear character on $R_{r,j} \wr S_s$.
	
	
\begin{lemma}\label{lem:char_values_and_inner_product}
		
		
    For every $0\le j\le r$,
    $f_j(r)$ is a non-negative integer and 
    \[ 
        \langle \phi_{r,s}\downarrow_{R_{r,j}\wr S_{s}}^{K_{r,s}},\nu_{r,j,s} \rangle 
        = (-1)^{(j+1)s} \binom{(-1)^{j+1} f_j(r)}{s}
        = \begin{cases}
            \binom{f_j(r)}{s}, &\text{if } j \text{ is odd;}\\
            \binom{f_j(r)+s-1}{s}, &\text{if } j \text{ is even.}\\
        \end{cases}
    \]
\end{lemma}
		
\begin{remark}
As a byproduct, Lemma~\ref{lem:char_values_and_inner_product}
provides a new proof of the non-negativity of $f_j(r)$.
Indeed, if $s=1$ then 
$\langle \phi_{r,1} \downarrow_{R_{r,j}}^{K_{r,1}},\nu_{r,j,1} \rangle = f_j(r)$ 
is clearly a non-negative integer.
\end{remark}	
	
The rest of this subsection 
consists of the proofs of Lemma~\ref{lem:char_values_and_inner_product} 
and Theorem~\ref{prop:e_i_formuli}. In these proofs 
$r$, $s$ and $j$ 
are fixed, unless specified otherwise. 
For convenience, we omit the indices and write 
$Z := Z_{(r^s)}$, 
$\omega := \omega^{(r^s)}$, 
$\psi := \psi^{(r^s)}$, 
$K := K_{r,s}$, 
$\phi := \phi_{r,s}$, 
$R := R_{r,j}$, and
$\nu := \nu_{r,j,s}$, 

\begin{proof}[Proof of  Lemma~\ref{lem:char_values_and_inner_product}.]
The proof consists of two parts. 
First we determine the character values of the induced character $\phi = \omega\uparrow_Z^K$ on the wreath product $K=S_r\wr S_{s}$; the resulting formula is Equation~\eqref{eq:char_value}. In the second part we 
apply this formula to compute the inner product.
		
Let $\zeta: \ZZ_r \to \mathbb{C}$ be the primitive linear character used to define $\omega$; see Definition~\ref{def:higher}. 
Recall the explicit formula for 
an induced character \cite[(5.1)]{Isaacs}: 
For a subgroup $H\leq G$ and a character $\chi$ of $H$, define $\chi^0 : G \to \mathbb{C}$ by $\chi^0(g) = \chi(g)$ if $g\in H$ and $\chi^0(g) = 0$ otherwise. Then
\begin{equation}\label{eq:induction}
    \chi\uparrow_H^G(y)
    = \frac{1}{|H|} \sum_{x\in G} \chi^0(x^{-1}yx)
    = \sum_{t\in T} \chi^0(t^{-1}yt),
\end{equation}
where $T$ is a full set of right coset representatives of $H$ in $G$.
		
An element of $K=S_r\wr S_s$ can be represented by an $s$-tuple of elements of $S_r$ and a wreathing permutation from $S_s$, so that
$K=\{(x_1,\ldots,x_s;\sigma) \mid x_1,\ldots,x_s \in S_{r},\,\sigma\in S_s\}$
with the product
$(x_1,\ldots,x_s;\sigma)(y_1,\ldots,y_s;\tau) = (x_1y_{\sigma^{-1}(1)},\ldots,x_sy_{\sigma^{-1}(s)};\sigma\tau)$. 
A full set of right coset representatives of $\ZZ_r$ in $S_r$ is $S_{r-1}$ (in the natural embedding).
Hence, a full set of right coset representatives of 
$Z = \ZZ_r \wr S_s$ in $K= S_r \wr S_s$ is 
$T = 
\{(x_1,\ldots,x_s;1)\mid (\forall i)\, x_i \in S_{r-1}\}$. For any $z_1,\ldots,z_s\in \ZZ_r$, the shifted set $(z_1,\ldots,z_s;1)T$ is also a full set of right coset representatives. 
Instead of taking the sum over $T$ in \eqref{eq:induction}, we will take it over the union of all the shifted sets, namely $\{(x_1,\ldots,x_s;1) \vert\, (\forall i)\, x_i \in S_r\}$, and divide by $r^s$. 
Since
\[
(x_1,\ldots,x_s;1)^{-1} (y_1,\ldots,y_s;\sigma) (x_1,\ldots,x_s;1)
= (x_1^{-1}y_1x_{\sigma^{-1}(1)},\ldots,x_s^{-1}y_sx_{\sigma^{-1}(s)};\sigma),
\]
we conclude that, for any $y = (y_1,\ldots,y_s;\sigma) \in K$,
\begin{align*}
    \phi(y) 
    &= \omega\uparrow_Z^K(y_1,\ldots,y_s;\sigma)
    = \frac{1}{r^s} \sum_{x_1,\ldots,x_s \in S_r} \omega^0(x_1^{-1}y_1x_{\sigma^{-1}(1)},\ldots,x_s^{-1}y_sx_{\sigma^{-1}(s)};\sigma) \\ 
    &= \frac{1}{r^s} \sum_{\substack{x_1,\ldots,x_s \in S_r \\ (\forall i)\, x_i^{-1}y_ix_{\sigma^{-1}(i)} \in \ZZ_r}} \omega(x_1^{-1}y_1x_{\sigma^{-1}(1)},\ldots,x_s^{-1}y_sx_{\sigma^{-1}(s)};\sigma) \\
    &= \frac{1}{r^s} \sum_{\substack{x_1,\ldots,x_s \in S_r \\ (\forall i)\, x_i^{-1}y_ix_{\sigma^{-1}(i)} \in \ZZ_r}} \prod_{i=1}^{s} \zeta(x_i^{-1}y_ix_{\sigma^{-1}(i)}).
\end{align*}

The factors in the product over $i$ are complex numbers, thus commute. We can therefore rearrange them in an order fitting the decomposition of $\sigma^{-1} \in S_s$ into disjoint cycles: if
\[
    \sigma^{-1} = C_1 \cdots C_t
\]
is a product of $t$ disjoint cycles, choose an element $a_k$ in each cycle $C_k$. Then
\[
    C_k = (a_k, \sigma^{-1}(a_k), \sigma^{-2}(a_k), \ldots)
    \qquad (1 \le k \le t).
\]
Since $\zeta$ is a linear character, cancellation gives
\begin{align*}
    \prod_{i \in C_k} \zeta(x_i^{-1}y_ix_{\sigma^{-1}(i)})
    &= \zeta(x_{a_k}^{-1}y_{a_k}x_{\sigma^{-1}({a_k})})
    \zeta(x_{\sigma^{-1}({a_k})}^{-1}y_{\sigma^{-1}({a_k})}x_{\sigma^{-2}(a_k)})
    \cdots \\
    &= \zeta(x_{a_k}^{-1}c_kx_{a_k}),
\end{align*}
where
\[
    c_k := y_{a_k} y_{\sigma^{-1}({a_k})} \cdots \in S_r
    \qquad (1 \le k \le t).
\]
		 
		
The condition $x_i^{-1}y_ix_{\sigma^{-1}(i)} \in \ZZ_r\, (\forall i)$ implies that the products $x_{a_k}^{-1}c_kx_{a_k}\in \ZZ_r\, (\forall k)$. 
Hence if $\phi(y) \ne 0$ then, necessarily, each cycle-product $c_k \in S_r$ must be conjugate to an element of $\ZZ_r$.
Since $\ZZ_r \le S_r$ is generated by a full cycle, 
a necessary and sufficient condition for $c_k$ to have a conjugate in $\ZZ_r$ is that it is a product of disjoint cycles of the same length.

For any divisor $d$ of $r$,
if $c_k \in S_r$ is a product of disjoint $d$-cycles
and $x_{a_k} \in S_r$ is such that $x_{a_k}^{-1}c_kx_{a_k} \in \ZZ_r$, 
then the value of $\zeta(x_{a_k}^{-1}c_kx_{a_k})$ is a primitive $d$-th root of unity. By varying the conjugating element $x_{a_k}$, each element of $\ZZ_r$ of order $d$ is obtained with the same multiplicity $|Z_{(d^{r/d})}| = (r/d)! d^{r/d}$.  
The other $x_i$'s, for $i \in C_k \setminus \{a_k\}$, are arbitrary, as long as $x_i^{-1}y_ix_{\sigma^{-1}(i)}\in \ZZ_r\,(\forall i)$. There are $r^{\ell_k - 1}$ such choices, where $\ell_k$ is the length of the cycle $C_k$.
We conclude that, for any $y \in K$ for which $c_k \in S_r$ is a product of disjoint $d_k$-cycles $(1 \le k \le t)$,
\[
    \phi(y) = \frac{1}{r^s} \prod_{k=1}^{t} 
    (r/d_k)! d_k^{r/d_k} r^{\ell_k - 1} 
    \sum_{z \in \ZZ_r \,:\, o(z) = d_k} \zeta(z).
\]
If $g$ is a generator of $\ZZ_r$, then $o(g^m) = d$ if and only if $m = jr/d$ for some integer $j$ coprime to $d$. 
It follows that
\[
    \sum_{z \in \ZZ_r \,:\, o(z) = d} \zeta(z) 
    = \sum_{0 \le j < d \,:\, (j,d) = 1} \zeta(g^{jr/d})
    = \mu(d).
\]
Since $\sum_{k=1}^{t} (\ell_k - 1) = s-t$,
we now have an explicit formula for the values of $\phi$:
\begin{equation}\label{eq:char_value}
	\phi(y) 
	= \begin{cases}\displaystyle
		\frac{1}{r^{t}} \prod_{k=1}^{t} \mu(d_k) d_k^{r/d_k} (r/d_k)!, 
		&\text{if } c_k \text{ is a product of disjoint } d_k \text{-cycles } (\forall k); \\
		0, &\text{otherwise.}
	\end{cases}
\end{equation}

To determine the inner product $\langle \phi\downarrow_{R\wr S_{s}}^K,\nu\rangle$ we evaluate the linear character $\nu$ on $R\wr S_{s}$.
Let $(v_1z_1,\ldots, v_sz_s; \sigma)\in R\wr S_s$, 
where $v_i\in S_j, z_i\in S_{r-j}$, $1\le i\le s$, and $\sigma\in S_s$. Then, by the definition of $\nu$
\[
\nu(v_1z_1,\ldots, v_sz_s; \sigma)
= \sgn(\sigma)^j \prod\limits_{i=1}^s \varepsilon(v_i)
= \sgn(\sigma)^j  \varepsilon(v_1,\ldots,v_s;1),
\]
where $\sgn(\sigma)$ denotes the sign of $\sigma \in S_s$.
We obtain
\begin{align*}
	\langle \phi\downarrow_{R\wr S_{s}}^K,\nu\rangle
	&= \frac{1}{|R\wr S_{s}|}
	\sum_{(v_1z_1,\ldots,v_sz_s;\sigma) \in R\wr S_{s}} \nu(v_1z_1,\ldots,v_sz_s;\sigma) \phi(v_1z_1,\ldots,v_sz_s;\sigma) \\
	&= \frac{1}{s!} \sum_{\sigma\in S_s} \sgn(\sigma)^j
	\frac{1}{(j!(r-j)!)^s} \sum_{\substack{v_1,\ldots ,v_s \in S_j \\ z_1,\ldots, z_s \in S_{r-j}}} \varepsilon(v_1,\ldots,v_s;1) \phi(v_1z_1,\ldots,v_sz_s;\sigma).
\end{align*}
For each nonzero summand and each cycle $C_k$ of $\sigma^{-1}$,
by Equation~\eqref{eq:char_value},
the cycle-product $c_k \in R \le S_r$ has cycle type $d_k^{r/d_k}$, and its restrictions to $S_j$ and $S_{r-j}$ have cycle types $d_k^{j/{d_k}}$ and $d_k^{(r-j)/{d_k}}$, respectively. 
In particular, $d_k | (r,j)$.
It follows that
\[
    \varepsilon(v_1,\ldots,v_s;1)
    = \prod_{k=1}^{t} (-1)^{(d_k+1)j/d_k}. 
\]
%
For any common divisor $d$ of $r$ and $j$, 
let $n_d$ denote the number of elements of $R$ which are products of disjoint $d$-cycles. 
If $C_k$ 
has length $\ell_k$, then 
the cycle-product $c_k$ can be a product of disjoint $d_k$-cycles in exactly $(j!(r-j)!)^{\ell_k-1}n_{d_k}$ ways.
For different cycles $C_k$ the choices of the respective $d_k$ are independent; the only restriction is $d_k|(r,j)$. 

Using again $\sum_{k=1}^{t} (\ell_k - 1) = s-t$ and Equation~\eqref{eq:char_value}, 
and denoting by $t = \cyc(\sigma)$ the number of cycles of $\sigma$, we obtain 
\begin{align*}
	\langle \phi\downarrow_{R\wr S_{s}}^K,\nu \rangle
	&= \frac{1}{s!} \sum_{\sigma\in S_s} \sgn(\sigma)^j
	\frac{1}{(j!(r-j)!)^s} \sum_{\substack{v_1,\ldots ,v_s \in S_j \\ z_1,\ldots, z_s \in S_{r-j}}} \varepsilon(v_1,\ldots,v_s;1) \phi(v_1z_1,\ldots,v_sz_s;\sigma) \\
	&= \frac{1}{s!} \sum_{\sigma\in S_s} \sgn(\sigma)^j
	\frac{1}{(j!(r-j)!)^{\cyc(\sigma)}} \sum_{d_1,\ldots, d_{\cyc(\sigma)} | (r,j)}\\
	&\, \qquad\qquad\qquad\qquad\qquad\qquad
	\frac{1}{r^{\cyc(\sigma)}} \prod_{k=1}^{\cyc(\sigma)} n_{d_k} (-1)^{(d_k+1)j/d_k} \mu(d_k) d_k^{r/d_k} (r/d_k)! \\
	&= \frac{1}{s!} \sum_{\sigma \in S_s} \sgn(\sigma)^j
	\frac{1}{(j!(r-j)!)^{\cyc(\sigma)}} \left( \sum_{d|(r,j)} \frac{\mu(d)d^{r/d}(r/d)!}{r} (-1)^{(d+1)j/d} n_d \right)^{\cyc(\sigma)}.
\end{align*}
Of course,
\[
    n_d
    = \frac{j!}{d^{j/d} (j/d)!} \cdot \frac{(r-j)!}{d^{(r-j)/d} ((r-j)/d)!}
    = \frac{j!(r-j)!}{d^{r/d}(j/d)!((r-j)/d)!}.
\]
Putting everything together, we obtain
\begin{align*} 
	\langle \phi\downarrow_{R \wr S_{s}}^K,\nu \rangle
	&= \frac{1}{s!} \sum_{\sigma \in S_s} \sgn(\sigma)^j 
	\left( \sum_{d|(r,j)} \frac{\mu(d)(-1)^{(d+1)j/d}}{r} \binom{r/d}{j/d} \right)^{\cyc(\sigma)} \\
	&= \frac{1}{s!} \sum_{\sigma \in S_s} \sgn(\sigma)^j f_j(r)^{\cyc(\sigma)},
\end{align*}
by Definition~\ref{def:f}.
		
In particular, if $s=1$ then $f = \langle \phi\downarrow_R^K,\nu \rangle$
is a non-negative integer.
		
Finally, by~\cite[Proposition 1.3.4]{EC1}, for any $s \ge 0$ and indeterminate $x$, 
\[
    \frac{1}{s!} \sum_{\sigma \in S_s} x^{\cyc(\sigma)}
    = \binom{x+s-1}{s}.
\]
Substituting $-x$ for $x$ and noting that 
$\sgn(\sigma) = (-1)^{s-\cyc(\sigma)}$, we get
\[
    \frac{1}{s!} \sum_{\sigma \in S_s} \sgn(\sigma) x^{\cyc(\sigma)} 
    = \binom{x}{s}.
\]
This yields the desired formula,
depending on the parity of $j$,
for $\langle \phi\downarrow_{R \wr S_{s}}^K,\nu \rangle$.
\end{proof}

To prove Theorem~\ref{prop:e_i_formuli} we need a final ingredient, a combinatorial parametrization of 
$(S_r\wr S_s,S_i\times S_{rs-i})$ double cosets of $S_{rs}$  by partitions.

Recall Definition~\ref{def:partition}. 
The idea and definition of $P_{r,s}(i)$ actually appear
already in~\cite{Giannelli}, 
in a similar context but without explicit reference to double cosets or Mackey's formula; 
see Definitions~2.8--2.10 and Proposition~2.11 there.

An example of a partition of $13$ representing a certain $(S_6\wr S_5,S_{13}\times S_{17})$ double coset appears in Figure~\ref{fig:subpart}.

\begin{lemma}\label{lemma:parametrization}
    Let $n = rs$, $K = K_{r,s} \cong S_r \wr S_s \le S_n$.
    There is a bijection between the $(K,S_i\times S_{n-i})$ double cosets of $S_n$ and $P_{r,s}(i)$, the set of partitions of $i$ into at most $s$ parts, all of size at most $r$. 
\end{lemma}

\begin{proof}
To describe the bijection from $K\backslash S_n/(S_i \times S_{n-i})$ to $P_{r,s}(i)$ explicitly,
first fix the underlying decomposition $\{1, \ldots, n\} = \{1,\ldots,i\} \cup \{i+1,\ldots, n\}$ for the action of $S_i\times S_{n-i}$. 
The left cosets in $S_n/(S_i \times S_{n-i})$ are clearly in bijection with the subsets of size $i$ in $\{1, \ldots, n\}$:
\[
g(S_i \times S_{n-i}) \longleftrightarrow g(\{1, \ldots, i\}).
\]
Now fix a decomposition for the action of $K \cong S_r \wr S_s$: 
$\{1, \ldots, n\} = B_1 \cup \ldots \cup B_s$,
where $B_j := \{(j-1)r+1,(j-1)r+2,\ldots,(j-1)r+r\}$ $(j = 1,\ldots,s)$. 
The elements of $S_s$ permute these blocks,
and each of the $s$ copies of $S_r$ acts on one of the blocks.
Given $g \in S_n$, we map the double coset $K g (S_i \times S_{n-i})$ to the partition $\gamma$ which is the non-increasing rearrangement of the sequence
\[
(\card{B_1 \cap g(\{1,\ldots,i\})}, \ldots, \card{B_s \cap g(\{1,\ldots,i\})}).
\]
This sequence consists of $s$ non-negative integers, each at most $r$, which sum up to $i$. Thus $\gamma \in P_{r,s}(i)$.
We will show that this map is a bijection.

For arbitrary $x\in K$ and $y\in S_i\times S_{n-i}$ we have 
\[
\card{B_j \cap xgy(\{1,\ldots,i\})} 
= \card{B_j \cap xg(\{1,\ldots,i\})}
= \card{x^{-1}(B_j) \cap g(\{1,\ldots,i\})}
\qquad (1 \le j \le s).
\]
The element $x^{-1} \in K = S_r \wr S_s$ permutes the blocks and permutes the elements of each block. This shows that the mapping $Kg(S_i \times S_{n-i}) \mapsto \gamma$ is well defined.

The mapping from $K\backslash S_n/(S_i \times S_{n-i})$ to $P_{r,s}(i)$ is clearly onto: 
for each $\gamma = (a_1,\ldots,a_s) \in P_{r,s}(i)$ there exists a permutation $g \in S_n$ such that  $\card{B_j \cap g(\{1,\ldots,i\})} = a_j$ $(1 \le j \le s)$. 
		
Finally, if $K g_1 (S_i \times S_{n-i})$ and $K g_2 (S_i \times S_{n-i})$ are mapped to the same partition $\gamma$, then there exists a permutation $\pi \in S_s$ satisfying
\[
\card{B_j \cap g_1(\{1,\ldots,i\})}
= \card{B_{\pi(j)} \cap g_2(\{1,\ldots,i\})}
\qquad (1 \le j \le s).
\]
Therefore there exist an element $x \in K = S_r \wr S_s$ such that
\[
B_j \cap g_1(\{1,\ldots,i\})
= B_j \cap xg_2(\{1,\ldots,i\})
\qquad (1 \le j \le s).
\]
It follows that
\[
g_1(\{1,\ldots,i\}) = xg_2(\{1,\ldots,i\}),
\]
and therefore $g_1 = xg_2y$ for a suitable permutation $y \in S_i \times S_{n-i}$.
\end{proof}

We are now ready to prove Theorem~\ref{prop:e_i_formuli}. 
The proof applies Lemma~\ref{lemma:parametrization} and the explicit bijection described in its proof, combined with Lemma~\ref{lem:char_values_and_inner_product}. 

\begin{proof}[Proof of  Theorem~\ref{prop:e_i_formuli}.]
Recall that $n=rs$, 
$K=K_{r,s}\cong S_r\wr S_s\leq S_n$ and 
$\psi = \phi\uparrow_K^{S_n}$. 
By Frobenius reciprocity (twice) and 
Mackey's formula~\cite[(5.2), Problem (5.6)]{Isaacs},
\begin{equation}\label{Mackey_sum}%
\begin{split}
    e_i 
    &= \langle \psi, \chi^{(1^i) \oplus (n-i)} \rangle 
    = \langle \phi\uparrow_K^{S_n}, (\varepsilon_{S_i} \times 1_{S_{n-i}})\uparrow_{S_i \times S_{n-i}}^{S_n} \rangle \\
    &= \langle {\phi\uparrow_K^{S_n}}\downarrow_{S_i \times S_{n-i}}^{S_n}, \varepsilon_{S_i}\times 1_{S_{n-i}} \rangle \\
    &= 
    \sum_{[g] \in K \backslash S_n / (S_i \times S_{n-i})}
    \langle \phi^g\downarrow^{K^g}_{K^g \cap (S_i \times S_{n-i})}\uparrow_{K^g \cap (S_i \times S_{n-i})}^{S_i \times S_{n-i}}, \varepsilon_{S_i} \times 1_{S_{n-i}} \rangle \\
    &=
    \sum_{[g] \in K \backslash S_n / (S_i \times S_{n-i})}
    \langle \phi^g\downarrow^{K^g}_{K^g \cap (S_i \times S_{n-i})}, (\varepsilon_{S_i} \times 1_{S_{n-i}})\downarrow^{S_i \times S_{n-i}}_{K^g \cap  (S_i \times S_{n-i})} \rangle.
\end{split}
\end{equation}
		
The above sums are indexed by the $(K,S_i\times S_{n-i})$ double cosets in $S_n$. 
For each representative $g$ of a double coset, 
$K^g := g^{-1}Kg$ is the corresponding conjugate of $K$ 
and $\phi^g$ is the character on $K^g$ defined by
$\phi^g(g^{-1}kg) := \phi(k)$ for all $k \in K$.


By Lemma~\ref{lemma:parametrization}, 
these double cosets are parametrized by the partitions in $P_{r,s}(i)$.  Let us determine the summand of \eqref{Mackey_sum} 
corresponding to a partition $\gamma\in P_{r,s}(i)$ in which part $j$ occurs with multiplicity $k_j = k_j(\gamma)$ $(0 \le j \le r)$. 
By the bijection described in the proof of Lemma~\ref{lemma:parametrization}, 
\[
    k_j = \card{\{t \,:\, \card{B_t \cap g(\{1, \ldots, i\})} = j\}}
    \qquad (0 \le j \le r),
\]
where $B_t := \{(t-1)r+1,\ldots,(t-1)r+r\}$ $(1 \le t \le s)$. Thus
\[
    K^g \cap (S_i \times S_{n-i})
    \cong (R_{r,0} \wr S_{k_0}) \times (R_{r,1} \wr S_{k_1}) \times \cdots \times (R_{r,r}\wr S_{k_r}),
\]
where $R_{r,j} = S_j \times S_{r-j}$, as above. In particular,
\begin{align*}
    (\varepsilon_{S_i} \times 1_{S_{n-i}})\downarrow^{S_i \times S_{n-i}}_{K^g \cap (S_i \times S_{n-i})}
    &= \bigotimes_j (\varepsilon_{S_{jk_j}} \times 1_{S_{(r-j)k_j}})\downarrow^{S_{jk_j} \times S_{(r-j)k_j}}_{R_{r,j}\wr S_{k_j}}
    = \bigotimes_j \nu_{r,j,k_j}.
\end{align*}
Note that, by Observation~\ref{obs:phi_product}, $\phi_{r,s}$ factors similarly.
Therefore the corresponding summand in~\eqref{Mackey_sum} is
\begin{equation}\label{Mackey_summand}
    \langle \phi^g\downarrow^{K^g}_{K^g \cap (S_i \times S_{n-i})},
    (\varepsilon_{S_i} \times 1_{S_{n-i}})\downarrow^{S_i \times S_{n-i}}_{K^g \cap (S_i \times S_{n-i})} \rangle 
    = \prod_{j=0}^{r} \langle \phi_{r,k_j}\downarrow^{S_r\wr S_{k_j}}_{R_{r,j}\wr S_{k_j}}, \nu_{r,j,k_j} \rangle.
\end{equation}

		
We have already computed these inner products in Lemma~\ref{lem:char_values_and_inner_product}.
By \eqref{Mackey_sum}, \eqref{Mackey_summand}, 
Lemma~\ref{lem:char_values_and_inner_product}
and Definition~\ref{def:f} we obtain
\[
    e_{i,(r^s)} 
    = \langle \phi\uparrow_K^{S_n}, (\varepsilon_{S_i} \times 1_{S_{n-i}})\uparrow_{S_i \times S_{n-i}}^{S_n} \rangle
    =
    \sum_{\gamma\in P_{r,s}(i)} \prod_{j=0}^{r} 
    (-1)^{(j+1)k_j(\gamma)} \binom{(-1)^{j+1} f_j(r)}{k_j(\gamma)},
\]
as claimed.
If $s=1$ then $P_{r,1}(i)$ contains (for $0 \le i \le r$) 
a unique partition $\gamma = (i)$, for which $k_i(\gamma) = 1$ is the unique nonzero multiplicity. Thus, in this case,  
\[
    e_{i,(r)} 
    = (-1)^{i+1} \binom{(-1)^{i+1}f_i(r)}{1} 
    = f_i(r).
\]
\end{proof}

\subsection{A product formula}
\label{sec:prod}
	
Now we derive a restatement of Theorem~\ref{prop:e_i_formuli} as a product formula for formal power series. 

\begin{corollary}\label{power_series_form}
For a positive integer $r$, define the formal power series
\[
E_r(x,y) := \sum_{i,s \ge 0} e_{i,(r^s)} x^i y^s.
\]
Then
\[
E_r(x,y) 
= \prod_{j=0}^{r} (1 - (-x)^j y)^{(-1)^{j+1}f_j(r)}.
\]
\end{corollary}
	
\begin{proof}
Recall the following formal power series expansion, valid for any integer $f$:
\[
(1+t)^{f} = \sum_{n=0}^\infty \binom{f}{n} t^n.
\]
By Theorem~\ref{prop:e_i_formuli},
with the obvious extension for $s=0$,
\begin{align*}
E_r(x,y) 
&=
\sum_{i,s \ge 0} \left(
\sum_{\substack{k_0,\ldots,k_r \ge 0 \\ \sum_j k_j=s \\ \sum_j jk_j = i}} 
\prod_{j = 0}^{r}
(-1)^{(j+1)k_j} \binom{(-1)^{j+1} f_j(r)}{k_j}
\right) x^i y^s \\
&= \prod_{j=0}^{r} \sum_{k_j=0}^\infty
(-1)^{(j+1)k_j} \binom{(-1)^{j+1} f_j(r)}{k_j}
x^{jk_j} y^{k_j} \\
&= \prod_{j=0}^{r} (1 + (-1)^{j+1} x^j y)^{(-1)^{j+1}f_j(r)},
\end{align*}
as required.
\end{proof}
	
\begin{remark}\label{rem:le_3}
For small $r$ and arbitrary $s$, Corollary~\ref{power_series_form}  enables us to determine explicitly the hook-multiplicities $m_{i,(r^s)}$. 
This is done by recalling Equation~\eqref{eq:alternating} and the fact that, by definition, $e_{i,(r^s)}$ is the coefficient of $x^i y^s$ in $E_r(x,y)$. 
For example, by Corollary~\ref{power_series_form} and Observation~\ref{obs:f_values}, 
$E_2(x,y) = E_3(x,y) = (1+xy)(1-x^2y)^{-1}$. 
Thus, for $r \in \{2,3\}$ and any $s \ge 1$, the value of $e_{i,(r^s)}$ is $1$ for $i \in \{2s-1,2s\}$ and zero otherwise. 
Combining this with Equation~\eqref{eq:alternating}, 
it follows that, for $r \in \{2,3\}$ and $s \ge 1$, the hook multiplicity $m_{i,(r^s)}$ 
is 1 for $i = 2s-1$ and zero otherwise. 
\end{remark}

\subsection{Non-negativity}\label{sec:nn}
		
Now we are ready to prove Theorem~\ref{t:main_hook-alternating2}.
	
\begin{proof}[Proof of Theorem~\ref{t:main_hook-alternating2}.]
Assume that $r$ is not square-free.
	Recall, from Definition~\ref{def:F_polynomial}, the notation $F_r(x) := \sum_{j=0}^{r} f_j(r) x^j$.
	By 	Proposition~\ref{prop:cycles_nonnegativity} and Corollary~\ref{t:F_and_M}	
    we may write $F_r(x) = (1+x)^2 G_r(x)$, where $G_r(x) = \sum_{j=0}^{r-2} g_j(r) x^j$ is a polynomial with non-negative integer coefficients. 
	Let $g_j(r) := 0$ for $j < 0$ or $j > r-2$. Then
	\[
	f_j(r) = g_j(r) + 2g_{j-1}(r) + g_{j-2}(r)
	\qquad (\forall j).
	\]
	%
    Therefore, by Corollary~\ref{power_series_form},
    \begin{equation*}\begin{split}
        E_r(x,y) 
        &= \prod_{j \ge 0} (1-(-x)^j y)^{(-1)^{j+1}f_j(r)} \\
        &= \prod_{j \ge 0} (1-(-x)^j y)^{(-1)^{j+1}(g_j(r) + 2g_{j-1}(r) + g_{j-2}(r))} \\
        &= \prod_{j \ge 0} \left(
        \frac{(1 - (-x)^{j} y)(1 - (-x)^{j+2} y)}{(1 - (-x)^{j+1} y)^2} \right)^{(-1)^{j+1} g_j(r)}.
    \end{split}
    \end{equation*}
    We claim that each factor in 
    this product has the form $1 + (1+x)^2 p_j(x,y)$, with $p_j(x,y)$ a formal power series with non-negative integer coefficients. 
    This implies that $(E_r(x,y)-1)/(1+x)^2$ is itself a formal power series with non-negative integer coefficients, completing the proof of Theorem~\ref{t:main_hook-alternating2}.
	
    Indeed, if $j$ is odd then the corresponding factor is
    \[
        \left( \frac{(1 + x^{j} y)(1 + x^{j+2} y)}{(1 - x^{j+1} y)^2} \right)^{g_j(r)}
        = \left( 1 + \frac{(x+1)^2 x^j y}{(1 - x^{j+1} y)^2} \right)^{g_j(r)},
    \]
    where
    \[
        \frac{x^j y}{(1 - x^{j+1} y)^2}
        = x^j y \cdot \sum_{i \ge 0} (i+1) (x^{j+1} y)^i
    \]
    is a formal power series with non-negative integer coefficients.
	
    Finally, if $j$ is even then the corresponding factor is
    \[
        \left( \frac{(1 + x^{j+1} y)^2}{(1 - x^{j} y)(1 - x^{j+2} y)} \right)^{g_j(r)}
        = \left( 1 + \frac{(x+1)^2 x^j y}{(1 - x^{j} y)(1 - x^{j+2} y)} \right)^{g_j(r)},
    \]
    where
    \[
        \frac{x^j y}{(1 - x^{j} y)(1 - x^{j+2} y)}
        = x^j y \cdot \sum_{i \ge 0} (x^{j} y)^i \cdot \sum_{k \ge 0} (x^{j+2} y)^k
    \]
    is a formal power series with non-negative integer coefficients.
    %
\end{proof}

\section{Additional results}
\label{sec:final}

\subsection{Combinatorial identities}
\label{remarks_on_s=1}

In this subsection it will be shown that Lemma~\ref{lem:char_eval_cycles} and Theorem~\ref{prop:e_i_formuli} imply well-known combinatorial identities.

\medskip

Recall the major index of a permutation $\pi\in S_n$,
\[
\maj(\pi):=\sum\limits_{i\in \Des(\pi)}i.
\]
The following identity is due to 
Garsia~\cite{Garsia}. A purely combinatorial proof was given by Wachs~\cite{Wachs}.

\begin{proposition}\label{prop:Garsia}~\cite[Equation 5.8]{Garsia}
     For every partition $\lambda\vdash n$,
    \[
        \sum_{\pi \in \C_\lambda} \zeta^{\maj(\pi)}
        = \begin{cases} 
            \mu(r), &\text{if } \lambda = (r^s); \\
    		0, &\text{otherwise },
		\end{cases}
    \]
    where $\zeta$ is a primitive $n$-th root of unity and $\mu$ is the M\"obius function. 
\end{proposition}

The following lemma follows from~\cite{Stembridge}.

\begin{lemma}\label{prop:equiv_Garsia}
    For every Schur-positive set $\A\subseteq S_n$ with associated $S_n$-character $\phi := \ch^{-1}(\Q(\A))$, the value of $\phi$ at an $n$-cycle $c\in S_n$ is
    \[
        \phi(c) = \sum_{\pi \in \A} \zeta^{\maj(\pi)},
    \]
    where $\zeta$ is a primitive $n$-th root of unity. 
\end{lemma}

\begin{proof}
First, recall 
the definition of the descent set of standard Young tableaux (SYT) 
from Equation~\eqref{def:Des_SYT}. By 
~\cite[Lemma 3.4]{Stembridge}, for every partition $\nu\vdash n$ the value of the irreducible $S_n$-character $\chi^\nu$ at an $n$-cycle $c\in S_n$ is 
\begin{equation}\label{eq:Stembridge}
\chi^\nu(c)=\sum\limits_{T\in \SYT(\nu)}\zeta^{\maj(T)}.
\end{equation}
Let $\A\subseteq S_n$ be Schur-positive with associated $S_n$-character $\phi$, i.e., 
$\Q(A)=\ch (\phi)$. 
By Lemma~\ref{lem:Schur-gf} 
together with Equation~\eqref{eq:Stembridge},
\[
    \sum_{\pi \in \A} \zeta^{\maj(\pi)}
    = \sum_{\nu \vdash n} \langle \Q(\A), s_\nu \rangle \sum_{T \in \SYT(\nu)} \zeta^{\maj(T)}
    = \sum_{\nu \vdash n} \langle \phi, \chi^\nu \rangle \chi^\nu(c)
    = \phi(c), 
\]
completing the proof.
\end{proof}

In light of Lemma~\ref{prop:equiv_Garsia} we deduce 

\begin{corollary}
    Lemma~\ref{lem:char_eval_cycles} is equivalent to Garsia's identity (Proposition~\ref{prop:Garsia}).
\end{corollary}

\begin{proof}
    By the Gessel-Reutenauer Theorem (Theorem~\ref{thm:GR}), for every $\lambda \vdash n$, 
    the conjugacy class of cycle type $\lambda$ is Schur-positive with $\ch^{-1}(\Q(\C_\lambda))=\psi^\lambda$. 
    Letting $\A = \C_\lambda$ in Lemma~\ref{prop:equiv_Garsia}, 
    Proposition~\ref{prop:Garsia} implies Lemma~\ref{lem:char_eval_cycles} and vice versa.
\end{proof}


\medskip

There is a combinatorial description, due to M.\ Schocker,
of the multiplicity of an arbitrary irreducible character of $S_n$ in the higher Lie character. 
In its full generality it is too complicated to be presented here, see~\cite{schocker} for details. 
The special case of a full cycle, $\lambda=(n)$ (for which
$\psi^{(n)} = \omega^{(n)}\uparrow_{\ZZ_n}^{S_n}$ is the Lie character), is due 
to Kra\'skiewicz and Weyman~\cite{KraskiewiczWeyman}. 
	
\begin{theorem}[Kra\'skiewicz-Weyman Theorem]\cite[Theorem 8.4]{Garsia}\label{thm:KW}
    For every partition $\nu\vdash n$, the multiplicity $m_{\nu,(n)} := \langle \psi^{(n)},\chi^{\nu} \rangle$ is equal to the cardinality of the set
    \[
        \{T \in \SYT(\nu):\ \maj(T) \equiv 1 \pmod n\}.
    \]	
\end{theorem}
	
\begin{corollary}\label{lem:Kra-Wey}
    For every $0\leq k\leq n$, the multiplicity 
    $m_{k,(n)} := \langle \psi^{(n)},\chi^{(n-k,1^k)} \rangle$ 
    is equal to the cardinality of the set
    \[
        \left\{ 1 \le a_1 < \cdots < a_k \le n-1 : \sum_{i=1}^k a_i \equiv 1 \!\!\!\!\pmod{n} \right\}.
    \]
\end{corollary}
	
\begin{proof}
    Denoting by $\binom{[n-1]}{k}$ the set of all $k$-subsets of $[n-1]$, 
    the map $\Des: \SYT(n-k,1^k) \rightarrow \binom{[n-1]}{k}$ is a bijection.
    The major index of a tableau is the sum of the elements of its descent set.
\end{proof}
	
Consider the following combinatorial identity.
\begin{proposition}\label{t:equivalence_KW}
    For every $0 \le k \le n$,
    \[
        \left| \left\{ 1 \le a_1 < \cdots < a_k \le n : \sum_{i=1}^k a_i\equiv 1 \!\!\!\!\pmod{n} \right\} \right| 
        = \sum_{d|(n,k)} \frac{\mu(d)(-1)^{(d+1)k/d}}{n} \binom{n/d}{k/d}.
    \]
\end{proposition}
	
\begin{proof}
    Let 
    $H(x,q):=\prod_{i=1}^n(1+xq^i)$. 
    Writing
    \[
        H(x,q) \equiv \sum_{k=0}^{n} \sum_{t=0}^{n-1} c_{k,t} x^kq^t \mod (q^n-1),
    \] 
    the cardinality we are interested in is the coefficient $c_{k,1}$ of $x^kq$.
    For any $d|n$ and $\eta$ a primitive $d$-th root of unity (we write $o(\eta)= d$), 
    \[
        H(x,\eta) 
        = \prod_{i=1}^n (1+x\eta^i) 
        = \left( \prod_{i=1}^d (1+x\eta^i) \right)^{n/d} 
        = (1-(-x)^d)^{n/d},
    \]
    where the last equality holds since both sides are polynomials of the same degree, with exactly the same roots and the same constant term. 
    %
    
    Let $\omega$ be a primitive $n$-th root of unity. Then
    \[
        H(x,\omega^j) = \sum_{t=0}^{n-1} h_t(x)\omega^{jt}
        \qquad (0 \le j \le n-1),
    \]
    where $h_t(x) = \sum_{k=0}^n c_{k,t} x^k$. 
    Fourier inversion gives
    \begin{align*}
        \sum_{k=0}^{n} c_{k,1}x^k 
        &= h_1(x)
        = \frac{1}{n} \sum_{j=0}^{n-1} H(x,\omega^j) \omega^{-j} \\
        &= \frac{1}{n} \sum_{d|n}  (1-(-x)^d)^{n/d}
        \sum_{j \,:\, o(\omega^j)=d} \omega^{-j} \\
        &= \frac{1}{n} \sum_{d|n}  (1-(-x)^d)^{n/d}
        \mu(d) \\
        &= \frac{1}{n} \sum_{k=0}^{n} \sum_{d|(n,k)} \binom{n/d}{k/d} (-1)^{(d+1)k/d} x^k \mu(d),
    \end{align*}
    as required.
\end{proof}

In light of  Corollary~\ref{lem:Kra-Wey} one observes

\begin{observation}\label{prop:equiv_f_KW}
    Proposition~\ref{prop:e_i_formuli1} is equivalent to Proposition~\ref{t:equivalence_KW}.
\end{observation}
		

\begin{proof}
    Comparing Corollary~\ref{lem:Kra-Wey} with Proposition~\ref{prop:e_i_formuli1} 
    yields
    \begin{align*}
        \left| \left\{ 1 \le a_1 < \cdots < a_k \le n : \sum_{i=1}^k a_i \equiv 1 \!\!\!\!\pmod{n} \right\} \right|
        &= m_{k,(n)} + m_{k-1,(n)} 
        = e_{k,(n)} 
        = f_k(n) \\
        &= \sum_{d|(n,k)} \frac{\mu(d)(-1)^{(d+1)k/d}}{n} \binom{n/d}{k/d}.
    \end{align*}
    The opposite direction is similar. 
\end{proof}

\begin{remark}
    Noting that Proposition~\ref{prop:e_i_formuli1} is the special case $s=1$ of Theorem~\ref{prop:e_i_formuli}, one concludes that Proposition~\ref{t:equivalence_KW} is a consequence of the latter.
\end{remark}
	

It remains a challenge to find such a direct link between Schocker's general description of the multiplicity and our version in Theorem~\ref{prop:e_i_formuli}.

\begin{remark}\label{rem:ET}
    Proposition~\ref{prop:e_i_formuli1} 
    is not new.
    For example, by the Gessel-Reutenauer Theorem (Theorem~\ref{thm:GR}) together with Observation~\ref{lem:1}, 
    Theorem~\ref{prop:e_i_formuli} at $s=1$ is equivalent to the following equation
    \[
        |\{\pi\in \C_{(n)}: \ \Des(\pi)=[j]\}|
        = \frac{1}{n} \sum_{d|n} \mu(d) (-1)^{j - \lfloor j/d \rfloor} \binom{n-1}{j-1}
        \qquad (0 \le j <n),        
    \]
    which is an immediate consequence of a recent result of Elizalde and Troyka~\cite[ Theorem 3.1]{ET}.  
    An older proof was presented to us by Sheila Sundaram~\cite{Sundaram_personal}, deducing 
    Proposition~\ref{prop:e_i_formuli1} 
    from~\cite[Lemma 2.7]{Sundaram_Adv}. 
    The reader is referred to~\cite{ET} for further discussion and relations to the enumeration of Lyndon words.  


%
%
\end{remark}


\subsection{Cellini's cyclic descents}\label{sec:Cellini}
	
In this subsection it is shown that the natural approach does not provide a cyclic descent extension for conjugacy classes in $S_n$.
	
\medskip
	
Recall the notion of {\em cyclic descent set} on permutations which was defined by Cellini~\cite{Cellini} 
\[
	\CDes(\pi) 
	:= \{1 \leq i \leq n \,:\, \pi_i > \pi_{i+1} \},
\]
with the convention $\pi_{n+1}:=\pi_1$.
	
Special subsets of $S_n$, for which Cellini's cyclic descent set map	is closed under cyclic rotation, thus determines a cyclic extension of $\Des$, are presented in~\cite{ER19}.
However, it seems that there are only two conjugacy classes ($2$-cycles in $S_3$ and $3$-cycles in $S_4$) for which Cellini's cyclic descent map 
is closed under cyclic rotation modulo $n$. Hence this cyclic extension does not serve our purposes.
	
It is unsettled whether there are other conjugacy classes, for which Cellini's $\CDes$ map is closed
under cyclic rotation.
Below are some partial results.
	
\begin{proposition}\label{prop:Cel1}
	For every $n>1$ and conjugacy class of cycle type $(r^s)$, ($n=rs$), Cellini's cyclic descent map is not is closed under cyclic rotation modulo $n$.
\end{proposition}
	
\begin{proof}
	Recall the notation $\C_\lambda$ from Section~\ref{sec:background}.
	For $r=1$, $\C_{(1^n)}$ consists of the identity permutation only, whose 
	Cellini's cyclic descent set is $\{n\}$. Thus not closed under rotation.
	For $r>1$ 
	let $\sigma=[s+1,s+2,\dots,n,1,2,\dots,s]$, in other words $\sigma$ is the permutation in $S_n$ defined by 
	\[
		\sigma(i)=i+s \pmod{n} \qquad (\forall i\in [n]).
	\]
	Then $\sigma$ is a product of $s$ disjoint cycles of length $r$ and Cellini's $\CDes$ of $\sigma$ is the singleton $\{n-s\}$. By equivariance property there must be a permutation of cycle type $(r^s)$ with cyclic descent set $\{n\}$. The only permutation in $S_n$ 
	with this property is the identity permutation, contradiction.
\end{proof}
	
\begin{proposition}
	For every conjugacy class of $k$-cycles, except $2$-cycles in $S_3$ and $3$-cycles in $S_4$, Cellini's cyclic descent map is not closed under cyclic rotation modulo $n$.
\end{proposition}
	
\begin{proof}
	Letting $r=n$ in Proposition~\ref{prop:Cel1}, statement holds on $n$-cycles.
	For $k<n$ let $\sigma\in \C_{(k,1^{n-k})}$ be the permutation $[k,1,2,\dots,k-1,k+1,k+2,\dots,n]$. Then $\CDes(\sigma)=\{1,n\}$.
	By the equivariance property,  there must be a $k$-cycle $\pi$ with cyclic descent set $\{1,2\}$. Then $\pi(3)$ is the minimal value thus it is equal to 1, and $\pi(1)$
	is the maximal value thus it is equal to $n$. Let $\pi(2)=x$, thus
	\[
		\pi=[n,x,1,2,\dots,x-1,x+1,\dots,n-1],
	\]
	namely, for every $3<i\le x+1$, $\pi(i)=i-2$ and for every $x+1<i\le n$, $\pi(i)=i-1$. Thus $\pi$ has no fixed points unless $x=2$, in this case $\pi$
	has cycle type $(n-1,1)$.  One deduces that statement holds for $k<n-1$.
		
		
	For the $(n-1)$-cycles an argument similar to the above works. For $n>4$ the permutation $\sigma=(1,n-1,n-2,\ldots,5,4,2,3)(n)$, that is, $[n-1,3,1,2,4,5,\ldots,n-2,n]\in \C_{(n-1,1)}$ has cyclic descent set $\CDes(\sigma)=\{1,2,n\}$. 
	By the equivariance property, there is a permutation $\pi\in \C_{(n-1,1)}$ with cyclic descent set $\{1,n-1,n\}$. Then $\pi(2)$ is the minimal value thus it is equal to $1$, and $\pi(n-1)$ is the maximal value thus it is equal to $n$. If $\pi(n)=l$ and $\pi(1)=k$ then $1<k<l<n$ and $\pi=[k,1,\dots,k-1,k+1,...,l-1,l+1,\dots,n,l]$, 
	so the cycle $(1,k,k-1,\ldots,2)$ of $\pi$ has length $2\le k<n-1$,  
	contradicting the equivariance property.
\end{proof}


\begin{conjecture} 
    Cellini's cyclic extension on a conjugacy class $\C$ is cyclic shift invariant only if $n \in \{3,4\}$ and $\C$ is the conjugacy class of $(n-1)$-cycles.
\end{conjecture}

\subsection{Palindromicity of hook multiplicities}
\label{sec:palindrom}	


A sequence $a_0,\ldots,a_n$ (equivalently, the polynomial $a_0+a_1x+\cdots+a_nx^n$) is 	
{\em palindromic} (or {\em symmetric})
if $a_i=a_{n-i}$ for all $0 \le i \le n$. 

For a partition $\lambda\vdash n$ recall the notation
\[
    m_{k,\lambda}
    := \langle \psi^{\lambda},\chi^{(n-k,1^k)} \rangle 
    \qquad (0\le k<n).
\]

In this subsection we prove

\begin{proposition}\label{conj:r2}
    Consider the partition $\lambda = (r^s)$ for positive integers $r$ and $s$.
    \begin{enumerate}
        \item 
        If $s=1$ then the hook-multiplicity sequence  $m_{0,(r)},\,m_{1,(r)},\ldots, m_{r-1,(r)}$ is palindromic if and only if either $r$ is odd or $r\equiv 0 \pmod 4$. 

        \item
        If $s>1$ then the hook-multiplicity sequence $m_{0,(r^s)},\,m_{1,(r^s)},\ldots, m_{rs-1,(r^s)}$ is palindromic if and only if $r\equiv 0 \pmod 4$.
    \end{enumerate}
\end{proposition}


\begin{proof}

{\bf 1.} 
Assume first that $s=1$.
Recall the notations $M_{(r)}(x) := \sum_{j=0}^{r-1} m_{j,(r)}x^j$ and 
$F_r(x) := \sum_{j=0}^{r} f_{j}(r)x^j$.  
By Corollary~\ref{t:F_and_M},
$F_r(x) = (1+x)M_{(r)}(x)$. 
Hence $M_{(r)}(x)$ is palindromic if and only if $F_r(x)$ is palindromic.  


If $r$ is odd then, for every $j$, every divisor $d|(r,j)$ is odd, hence $(d+1)j/d$ 
is even. 
Thus, by Definition~\ref{def:f}, 
\[
    f_j(r)
    = \sum_{d|(r,j)} \frac{\mu(d)}{r} \binom{r/d}{j/d} 
    = \sum_{d|(r,r-j)} \frac{\mu(d)}{r} \binom{r/d}{(r-j)/d}
    = f_{r-j}(r) 
    \qquad (0 \le j \le r),
\]
 and $F_r(x)$ is palindromic.
 
Next consider the case $r\equiv 0 \pmod 4$. 
Since $\mu(d)=0$ for $d\equiv 0 \pmod 4$, Definition~\ref{def:f} implies that  
\begin{align*}
    f_j(r)
    &= \sum_{d|(r,j)} 
    \frac{\mu(d)(-1)^{(d+1)j/d}}{r} \binom{r/d}{j/d} \\
    &= \sum_{\substack{d|(r,j) \\ d \text{ odd}}}
    \frac{\mu(d)(-1)^{(d+1)j/d}}{r} \binom{r/d}{j/d}
    + \sum_{\substack{d|(r,j) \\ d \equiv 2\!\!\!\! \pmod 4}}
    \frac{\mu(d)(-1)^{(d+1)j/d}}{r} \binom{r/d}{j/d}.
\end{align*}
Again, if $d|(r,j)$ is odd then $(d+1)j/d$ is even. 
If $d \equiv 2 \pmod 4$ then
\[
    (-1)^{(d+1)j/d} (-1)^{(d+1)(r-j)/d}
    = (-1)^{(d+1)r/d} = 1,
\]
since $r/d$ is even. 
It follows that
\[
    (-1)^{(d+1)j/d} = (-1)^{(d+1)(r-j)/d}
\]
in this case. 
We deduce that if $r\equiv 0\pmod 4$ then, for every $0 \le j \le r$,
\begin{align*}
    f_j(r) 
    &= \sum_{\substack{d|(r,j) \\ d \text{ odd}}} 
    \frac{\mu(d)}{r} \binom{r/d}{j/d} 
    + \sum_{\substack{d|(r,j) \\ d \equiv 2\!\!\!\! \pmod 4}}
    \frac{\mu(d)(-1)^{(d+1)j/d}}{r} \binom{r/d}{j/d} \\
    &= \sum_{\substack{d|(r,r-j) \\ d \text{ odd}}}
    \frac{\mu(d)}{r} \binom{r/d}{(r-j)/d}
    + \sum_{\substack{d|(r,r-j) \\ d \equiv 2\!\!\!\! \pmod 4}}
    \frac{\mu(d)(-1)^{(d+1)(r-j)/d}}{r} \binom{r/d}{(r-j)/d} 
    = f_{r-j}(r),  
 \end{align*}
 hence $F_r(x)$ is palindromic.



On the other hand, if 
$r\equiv 2 \pmod 4$ then,  
letting $j = 2$,
\[
    f_2(r) 
    = \frac{1}{r} \left[ \binom{r}{2} + \binom{r/2}{1} \right]
    = \frac{r}{2}
    \ne \frac{r-2}{2}
    = \frac{1}{r} \left[ \binom{r}{r-2} - \binom{r/2}{(r-2)/2} \right]
    = f_{r-2}(r),
\]
thus $F_r(x)$ is not palindromic in this case.

\noindent 
{\bf 2.}  
Assume now that $s > 1$.
By Equation~\eqref{eq:alternating}, 
$\sum_{j=0}^{rs} e_{j,(r^s)}x^j
= (1+x) \sum_{j=0}^{rs-1} m_{j,(r^s)}x^j$.
Thus the sequence $m_{0,(r^s)}, \ldots m_{rs-1,(r^s)}$ is palindromic if and only if the sequence $e_{0,(r^s)}, \ldots, e_{rs,(r^s)}$ is.
We will show that this happens if and only if $r \equiv 0 \pmod 4$.

Assume first that $r \equiv 0 \pmod 4$.
Following Definition~\ref{def:partition},
for each partition 
$\gamma  = (\gamma_1, \ldots, \gamma_s) \in P_{r,s}(i)$ 
consider the complementary partition 
$\bgamma = (\bgamma_1, \ldots, \bgamma_s) \in P_{r,s}(rs-i)$, defined by  
$\bgamma_\ell := r-\gamma_{s+1-\ell}$ $(1 \le \ell \le s)$,
or equivalently by
$k_j(\bgamma) = k_{r-j}(\gamma)$ $(0 \le j \le r)$.
Since $r \equiv 0 \pmod 4$,
$f_j(r) = f_{r-j}(r)$ $(0 \le j \le r)$,
by Part 1 of the current proof.
In addition, $j+1$ and $r-j+1$ have the same parity when $r$ is even and $j$ is arbitrary. 
Using Proposition~\ref{prop:e_i_formuli},
it follows that for $r\equiv 0 \pmod 4$ and any $0 \le i \le rs$,
\begin{align*}
    e_{i,(r^s)}
    &= \sum_{\gamma \in P_{r,s}(i)} \prod_{j=0}^{r}
    (-1)^{(j+1)k_j(\gamma)} \binom{(-1)^{j+1} f_j(r)}{k_j(\gamma)} \\
    &= \sum_{\gamma \in P_{r,s}(rs-i)} \prod_{j=0}^{r}
    (-1)^{(r-j+1)k_{r-j}(\gamma)} \binom{(-1)^{r-j+1} f_{r-j}(r)}{k_{r-j}(\gamma)}= e_{rs-i,(r^s)},  
\end{align*}
proving palindromicity in this case. 		

For the converse,
    consider again the explicit formula for $e_{i,(r^s)}$ (from Proposition~\ref{prop:e_i_formuli}) written above. 
    Each summand corresponds to a partition $\gamma \in P_{r,s}(i)$. 
    According to Remark~\ref{rem:when_factor_is_zero}, the summand is zero if and only if either
    $k_j(\gamma) > f_j(r)$ for some odd $j$, or
    $k_j(\gamma) > 0 = f_j(r)$ for some even $j$.
    
        By Observation~\ref{obs:f_values}, 
        $f_0(r) = 0$ for $r>1$, 
        $f_r(r) = 0$ for $r>2$, and
        $f_1(r) = f_{r-1}(r) = 1$ for $r>0$.
        It follows that for $\gamma \in P_{r,s}(i)$ to contribute a nonzero summand, it is necessary that
        $k_0(\gamma) = 0$ (for $r>1$), 
        $k_r(\gamma) = 0$ (for $r>2$), 
        $k_1(\gamma) \in \{0,1\}$ (for $r>0$) and 
        $k_{r-1}(\gamma) \in \{0,1\}$ (for $r-1$ odd).
        


        Assume first that $r>1$ is odd. 
        The restrictions on $k_0(\gamma)$ and $k_1(\gamma)$ imply that for $i=s$ there is no relevant $\gamma \in P_{r,s}(s)$
        (since $s > 1$). 
        The restriction on $k_r(\gamma)$ implies that for $i=rs-s$ there is only one relevant $\gamma \in P_{r,s}(rs-s)$, with $k_{r-1}(\gamma) = s$ (and $k_j(\gamma) = 0$ for all other values of $j$).
        Thus 
        \[
            e_{rs-s,(r^s)}
            = \binom{s}{s}
            = 1 > 0
            = e_{s,(r^s)},
        \]
        and there is no palindromicity.
        
        If $r=1$ then $f_0(1) = f_1(1) = 1$ and the unique partition $\gamma \in P_{1,s}(i)$ has $k_1(\gamma) = i$ and $k_0(\gamma) = s-i$ $(0 \le i \le s)$. It follows that
        \[
            e_{i,(1^s)} 
            = \binom{s-i}{s-i} \binom{1}{i}
            = \begin{cases}
                1, &\text{if } i \in \{0,1\}; \\
                0, &\text{otherwise.}
            \end{cases}
        \]
        For $s>1$ this sequence is not palindromic.
        
        Assume now that $2 < r \equiv 2 \pmod 4$.
        The restrictions on $k_0(\gamma)$ and $k_1(\gamma)$ imply (actually, for any $r>1$) that 
        $e_{i,(r^s)}=0$ for $0 \le i < 2s-1$ and 
        $e_{2s-1,(r^s)} = \binom{f_2(r)+s-2}{s-1}$.
        Similarly, the restrictions on $k_r(\gamma)$ and $k_{r-1}(\gamma)$ imply (for even $r>2$) that 
        $e_{rs-i,(r^s)}=0$ for $0 \le i < 2s-1$ and 
        $e_{rs-2s+1,(r^s)} = \binom{f_{r-2}(r)+s-2}{s-1}$.  
Recall, from Part 1 of the current proof, that for $r \equiv 2 \pmod 4$
\[
    f_2(r) 
    = \frac{r}{2} 
    > \frac{r-2}{2} 
    = f_{r-2}(r).
\]
Therefore
\[
    e_{2s-1,(r^s)}
    = \binom{f_2(r)+s-2}{s-1} 
    > \binom{f_{r-2}(r)+s-2}{s-1}
    = e_{rs-2s+1,(r^s)}
\]
and the sequence is not palindromic.

Finally, if $r=2$ then,
by Remark~\ref{rem:le_3},
\[
    e_{i,(2^s)} 
    = \begin{cases}
        1, &\text{if } i \in \{2s-1,2s\}; \\
        0, &\text{otherwise}
    \end{cases}
\]
and this sequence is not palindromic.
This completes the proof. 
\end{proof}

\section{Final remarks and open problems}\label{sec:open}











Recall the notation $m_{k,\lambda}:= \langle 
\psi^{\lambda},\chi^{(n-k,1^k)} \rangle$. 
By Proposition~\ref{prop:s=1_unimodal}, the \emph{hook-multiplicity sequence} $m_{0,(n)},\,m_{1,(n)},\ldots, m_{n-1,(n)}$ is unimodal; 
We conjecture that it is unimodal for all partitions $\lambda\vdash n$.  

 
    \begin{conjecture}\label{conj:unimodal}
    For every partition $\lambda \vdash n$, the sequence
    \[
    m_{0,\lambda},m_{1,\lambda},\ldots,m_{n-1,\lambda}
    \]
    is unimodal.
    \end{conjecture}
    
 The conjecture has been verified for all partitions of size $n \le 15$ and for all partitions of rectangular shape $(r^s)$ with $r \le 40$ and $s \le 5$.   
    

Note that, by Lemma~\ref{t:unimodal}, Conjecture~\ref{conj:unimodal} can provide an alternative proof of Theorem~\ref{t:main_hook-alternating2}. 
 
\medskip 
 
	A sequence $a_0,\ldots,a_n$ of real numbers is 
	\emph{log-concave} if  $a_{i-1}a_{i+1}\le a_i^2$ for all  $0<i<n$. 
	It is not hard to show that a log-concave sequence with no internal zeros is unimodal, see e.g.~\cite{Brenti, Stanley}.
Since log-concavity implies unimodality, it is tempting to check whether the hook-multiplicity sequence is log-concave.

\begin{conjecture}\label{conj:r1}
For every partition $\lambda=(r^s)$ with even $r\ne 6$,
the hook-multiplicity sequence $m_{0,\lambda},\,m_{1,\lambda},\ldots, m_{n-1,\lambda}$ is log-concave.
\end{conjecture}


Conjecture~\ref{conj:r1} 
was checked for all $r\leq 40$ and $s\leq 5$.

\bigskip

Our proof of Theorem~\ref{thm:main} is not constructive. 
We conclude the paper with the following challenging 
problem.

\begin{problem} 
    Find 
    an explicit combinatorial description of the cyclic descent extension on the conjugacy class of cycle type $\mu$, not equal to $(r^s)$ for a square-free $r$.
\end{problem}

A solution of this problem for the conjugacy classes of involutions 
is presented in~\cite{Adin-R23}. The analogous problem for standard Young tableaux of fixed non-ribbon shape was solved by Huang~\cite{Huang}. 


\end{document}